\documentclass{amsart}
\usepackage{amssymb}
\usepackage{latexsym}
\usepackage{amssymb}
\usepackage{amsmath}
\usepackage{mathrsfs}
\usepackage{amsbsy}
\usepackage{cite}
\usepackage{hyperref}
\usepackage[latin1]{inputenc}
\usepackage{tikz}

\usetikzlibrary{shapes,arrows}
\tikzstyle{decision} = [diamond, draw, text width=6em, text badly centered, node distance=3cm, inner sep=0pt]
\tikzstyle{block} = [rectangle, draw, text width=5em, text centered, rounded corners, minimum height=4em]
\tikzstyle{block1} = [rectangle, draw, text width=8em, text centered, rounded corners, minimum height=4em]
\tikzstyle{line} = [draw, -latex']
\tikzstyle{cloud} = [draw, ellipse, node distance=3cm, minimum height=2em]
\theoremstyle{plain}
\newtheorem{theorem}{Theorem}[section]

\newtheorem{lemma}[theorem]{Lemma}

\newtheorem{question}[theorem]{Question}
\newtheorem{proposition}[theorem]{Proposition}

\theoremstyle{definition}
\newtheorem{criterion}[theorem]{Criterion}

\newtheorem{definition}[theorem]{Definition}

\begin{document}
\title[Unitary equivalence of automorphisms of separable C*-algebras]{%
Unitary equivalence of automorphisms\\
of separable C*-algebras}
\author{Martino Lupini}
\address{Department of Mathematics and Statistics, York University, Canada.}
\email{mlupini@mathstat.yorku.ca}
\thanks{The research was supported by the York University Elia Scholars
Program and by the Oberwolfach Leibniz Graduate Student Programme.}

\begin{abstract}
We prove that the automorphisms of any separable C*-algebra that does not
have continuous trace are not classifiable by countable structures up to
unitary equivalence. This implies a dichotomy for the Borel complexity of
the relation of unitary equivalence of automorphisms of a separable unital
C*-algebra: Such relation is either smooth or not even classifiable by
countable structures.
\end{abstract}

\maketitle

\section{Introduction}

\label{Section: Introduction}

If $A$ is a separable C*-algebra, the group $\mathrm{Aut}(A)$ of
automorphisms of $A$ is a Polish group with respect to the topology of
pointwise norm convergence. An automorphism of $A$ is called (multiplier) 
\textit{inner} if it is induced by the action by conjugation of a unitary
element of the multiplier algebra $M(A)$ of $A$. Inner automorphisms form a
Borel normal subgroup $\mathrm{Inn}(A)$ of the group of automorphisms of $A$%
. The relation of\textit{\ unitary equivalence }of automorphisms of $A$ is
the coset equivalence relation on $\mathrm{Aut}(A)$ determined by $\mathrm{%
Inn}(A)$. (The reader can find more background on C*-algebras in Section \ref%
{Section: background}.) The main result presented here asserts that if $A$
does not have continuous trace, then it is not possible to effectively
classify the automorphisms of $A$ up to unitary equivalence using countable
structures as invariants; in particular this rules out classification by
K-theoretic invariants. In the course of the proof of the main result we will show that the
existence of an outer derivation on a C*-algebra $A$ is equivalent to a
seemingly stronger statement, that we will refer to as Property AEP (see
Definition \ref{Definition: property AEP}), implying in particular the
existence of an outer derivable automorphism of $A$.

The notion of effective classification can be made precise by means of Borel
reductions in the framework of descriptive set theory (the monographs \cite%
{Kechris} and \cite{Gao} are standard references for this subject). If $E$
and $E^{\prime }$ are equivalence relations on standard Borel spaces $X$ and 
$X^{\prime }$ respectively, then a Borel reduction from $E$ to $F$ is a Borel function $f:X\rightarrow X^{\prime }$ such that for every
$x$,$y\in X$, $xEy$ if and only if $f\left( x\right) E^{\prime }f\left(
y\right) $. The Borel function $f$ witnesses an \textit{effective
classification} of the objects of $X$ up to $E$, with $E^{\prime }$%
-equivalence classes of objects of $X^{\prime }$ as invariants. This
framework captures the vast majority of concrete classification results in
mathematics. (In \cite{FTT1} and \cite{FTT2} the computation of most of the
invariants in the theory of C*-algebras is shown to be Borel.)

If $E$ and $F$ are, as before, equivalence relations on standard Borel spaces,
then $E$ is defined to be \textit{Borel reducible} to $F$ if there is a Borel
reduction from $E$ to $F$. This can be interpreted as a notion that
allows one to compare the complexity of different equivalence relations.
Some distinguished equivalence relations are used as benchmarks of
complexity. Among these are the relation $=_{Y}$ of equality for elements of
a Polish space $Y$, and the relation $\simeq _{\mathcal{C}}$ of isomorphism
within some class of countable structures $\mathcal{C}$. If $E$ is an
equivalence relation on a standard Borel space $X$, we say that:

\begin{itemize}
\item $E$ is \textit{smooth} (or the elements of $X$ are \textit{concretely
classifiable} up to $E$) if $E$ is Borel reducible to $=_{Y}$ for some
Polish space $Y$;

\item $E$ is \textit{classifiable by countable structures} (or the elements
of $X$ are classifiable by countable structures up to $E$) if $E$ is Borel
reducible to $\simeq _{\mathcal{C}}$ for some class $\mathcal{C}$ of
countable structures.
\end{itemize}

A nontrivial example of smooth equivalence relation is the relation of
unitary equivalence of irreducible representations of a Type I C*-algebra (see \cite{Blackadar} Definition
IV.1.1.1).
Since all uncountable Polish spaces are Borel isomorphic to $\mathbb{R}$,
the class of smooth equivalence relations includes only the equivalence
relations that are effectively classifiable using real numbers as invariants. The class
of equivalence relations that are classifiable by countable structures is
much wider. In fact most classification results in mathematics involve some
class of countable structures as invariants. Elliott's seminal
classification of AF algebras by the ordered $K_{0}$ group in \cite%
{Elliott:classificationAF} is of this sort, as well as the K-theoretical
classification of purely infinite simple nuclear C*-algebras in the UCT
class obtained by Kirchberg and Phillips in \cite{Kirchberg:congress} and 
\cite{Phi:classification}. Nonetheless, in the last decade a number of
natural equivalence relations arising in different areas of mathematics have
been shown to be not classifiable by countable structures. A key role in
this development has been played by the theory of turbulence, developed by
Greg Hjorth in the second half of the 1990s.

Turbulence is a dynamic condition on a continuous action of a Polish group
on a Polish space, implying that the associated orbit equivalence relation
is not classifiable by countable structures. Many nonclassifiability results
were established directly or indirectly using this criterion. For instance
Hjorth showed in \cite{Hjorth-book} (Section 4.3) that the orbit equivalence
relation of a turbulent Polish group action is Borel reducible to the
relation of homeomorphism of compact spaces, which in turn is reducible to
the relation of isomorphism of separable simple nuclear unital C*-algebras
by a result of Farah-Toms-Törnquist (Corollary 5.2 of \cite{FTT1}). As a
consequence these equivalence relations are not classifiable by countable
structures.

In this paper, we use Hjorth's theory of turbulence to prove the following
theorem:

\begin{theorem}
\label{Theorem: main}If $A$ is a separable C*-algebra that does not have
continuous trace, then the automorphisms of $A$ are not classifiable by
countable structures up to unitary equivalence.
\end{theorem}

Theorem \ref{Theorem: main} strengthens Theorem 3.1 from \cite{Phillips1},
where the automorphisms of $A$ are shown to be not concretely classifiable
under the same assumptions on the C*-algebra $A$. We will in fact show that
the same conclusion holds even if one only considers the subgroup consisting
of approximately inner automorphisms of $A$, i.e.\ pointwise limits of inner
automorphisms. A C*-algebra has continuous trace (see Definition IV.1.4.12
and Proposition IV.1.4.19 of \cite{Blackadar}) if it has Hausdorff spectrum
and it is generated by its abelian elements. The class of C*-algebras that
do not have continuous trace is fairly large, and in particular includes
C*-algebras that are not Type I. More information about C*-algebras with continuous trace can be found in the monograph \cite{Raeburn-Williams}.

A particular implication of Theorem \ref{Theorem: main} is that it is not possible to
classify the automorphisms of any separable C*-algebra that does not have
continuous trace up to unitary equivalence by K-theoretic invariants. This
should be compared with the classification results of (sufficiently outer)
automorphisms up to other natural equivalence relations, such as outer
conjugacy (see\ \cite{Nakamura} Section 3). Nakamura showed in \cite%
{Nakamura} (Theorem 9) that aperiodic automorphisms of Kirchberg algebras
are classified by their KK-classes up to outer conjugacy. Theorem 1.4 of 
\cite{Kishimoto1} asserts that there is only one outer conjugacy class of
uniformly aperiodic automorphisms of UHF algebras. These results were more
recently generalized and expanded to classification of actions of $\mathbb{Z}%
^{2}$ and $\mathbb{Z}^{n}$ up to outer conjugacy or cocycle conjugacy (see\ 
\cite{Matui2}, \cite{Matui3}, \cite{Katsura-Matui}, and \cite{Matui-Sato1}).

Phillips and Raeburn obtained in \cite{Phillips-Raeburn} a cohomological
classification of automorphisms of a C*-algebra with continuous trace up to
unitary equivalence. Such classification implies that if $A$ has continuous trace and the spectrum of $A$ is homotopy equivalent to a compact
space, then the normal subgroup $\mathrm{\mathrm{Inn}}(A)$ of inner
automorphisms is closed in $\mathrm{Aut}(A)$ (see Theorem 0.8 of \cite%
{Raeburn-Rosenberg}). In particular (cf.\ Corollary II.6.5.7 in \cite%
{Blackadar}) this conclusion holds when $A$ is unital and has continuous
trace. It follows from a standard result in descriptive set theory (see\
Exercise 6.4.4 of \cite{Gao}) that the automorphisms of $A$ are concretely
classifiable up to unitary equivalence if and only if $\mathrm{Inn}(A)$ is a
closed subgroup of $\mathrm{Aut}(A)$. Theorem 0.8 of \cite%
{Raeburn-Rosenberg} and Theorem \ref{Theorem: main} therefore imply the following
dichotomy result:

\begin{theorem}
\label{Theorem: unital} If $A$ is a separable unital C*-algebra, then the
following statements are equivalent:

\begin{enumerate}
\item the automorphisms of $A$ are concretely classifiable up to unitary
equivalence;

\item the automorphisms of $A$ are classifiable by countable structures up
to unitary equivalence;

\item $A$ has continuous trace.
\end{enumerate}
\end{theorem}

More generally the same result holds if $A$ is a separable C*-algebra with
(not necessarily Hausdorff) compact spectrum. Without this hypothesis the
implication $3\Rightarrow 1$ of Theorem \ref{Theorem: unital} does not hold,
as pointed out in Remark 0.9 of \cite{Raeburn-Rosenberg}. We do not know if
the implication $3\Rightarrow 2$ holds for a not necessarily unital
C*-algebra $A$. This is commented on more extensively in Section \ref%
{Section: open problems}.

In particular Theorem \ref{Theorem: unital} offers another
characterization of unital C*-algebras that have continuous trace, in
addition to the classical Fell-Dixmier spectral condition (see \cite{Fell}, 
\cite{Dixmier-traces}) or the reformulation in terms of central sequences by
Akemann and Pedersen (see \cite{Akemann-Pedersen} Theorem 2.4).

The dichotomy in the Borel complexity of the relation of unitary equivalence
of automorphisms of a unital C*-algebra expressed by Theorem \ref{Theorem:
unital} should be compared with the analogous phenomenon concerning the
relation of unitary equivalence of irreducible representations of a
C*-algebra $A$. It is a classical result of Glimm from \cite{Glimm} that
such a relation is smooth if and only if $A$ is Type I. It was proved in \cite%
{Kerr-Li-Pichot} and, independently, in \cite{Farah-Mackey-Borel} that the
irreducible representations of a C*-algebra that is not Type I are in fact
not classifiable by countable structures up to unitary equivalence.


The strategy of the proof of Theorem \ref{Theorem: main}, summarized in
Figure \ref{figure}, is the following: We first introduce properties AEP and AEP$^{+}$, named after Akemann, Elliott, and Pedersen since they can be found \textit{in nuce} in their works \cite{Akemann-Pedersen} and \cite%
{Elliott3}. We then show that Property AEP$^{+}$ is stronger than Property AEP; moreover Property AEP is
equivalent to the existence of an outer derivation, 
and it implies that the conclusion of Theorem \ref{Theorem: main} holds. 
This concludes the proof under
the assumption that the C*-algebra $A$ has an outer derivation. We then
assume that $A$ does not have continuous trace and has only inner
derivations. Using a characterization of C*-algebras with only inner
derivations due to Elliott (the main theorem in \cite{Elliott3}) and a
characterization of continuous trace C*-algebras due to Akemann-Pedersen
(Theorem 2.4 in \cite{Akemann-Pedersen}), we infer that in this case $A$ has
a simple nonelementary direct summand. We then deduce 
that $A$ contains a central sequence
that is not hypercentral (a similar result was proved by Phillips in the
unital case, cf. Theorem 3.6 of \cite{Phillips2}). The proof is finished by
proving that the existence of a central sequence that is not hypercentral
implies that the conclusion of Theorem \ref{Theorem: main} holds.

\begin{figure}[h]
\par
\begin{tikzpicture}[node distance = 2cm, auto]

		\node [cloud] (noncontinuous) {non continuous trace};

		\node [decision, below of=noncontinuous, node distance=3cm] (question) {outer derivations?\\ \,};

	\node [block, below of=question, left of=question, node distance=2cm](outer) {outer derivation};

	\node[block, left of=outer, node distance=3cm] (AEP) {Property AEP};
	
	\node[block, above of=AEP] (AEP+) {Property AEP$^{+}$};

		\node [block, below of=question, node distance=2cm, right of=question] (inner) {only inner derivations};

	\node[block1, below of=inner] (simple) {simple\\ nonelementary\\ direct summand};

	\node [block1, below of=simple] (central) {central\\ nonhypercentral\\ sequence};

	\node [cloud, left of=central, node distance=5cm] (nonclassification) {nonclassification};
  
	
		
   \path [line] (noncontinuous) -- (question);

\path[double] (outer) edge[<->](AEP); 

   \path [line] (AEP+) -- (AEP);

   \path [line] (AEP) -- (nonclassification);

   \path [line] (inner) -- (simple);

   \path [line] (simple) -- (central);

   \path [line] (central) -- (nonclassification);

    \path [line, dashed] (question)  -| node [above, near start] {yes} (outer);

    \path [line, dashed] (question) -| node [above, near start] {no} (inner); 


\end{tikzpicture}
\caption{}
\label{figure}
\end{figure}

This paper is organized as follows:\ Section \ref{Section: background}
contains some background on C*-algebras and introduces the notations used in
the rest of the paper; Section \ref{Section: criteria} infers from
Hjorth's theory of turbulence a criterion of nonclassifiability by countable
structures (Criterion \ref{Criterion: non-classifiability}), to be applied
in the proof of Theorem \ref{Theorem: main}; 
Section \ref{Section: outer} establishes Theorem \ref{Theorem: main} in the case of C*-algebras with outer derivations, while Section \ref{Section: inner} deals with the case of C*-algebras with only innner derivations;
Section \ref{Section: dichotomy derivations} present a
dichotomy result for derivations analogous to Theorem \ref{Theorem: unital}
(Theorem \ref{Theorem: dichotomy derivations}). We conclude in Section \ref{Section: open problems} with some
remarks and open problems.

\section{Background on C*-algebras and notation}

\label{Section: background}A C*-algebra is a norm-closed self-adjoint
subalgebra of the Banach *-algebra $B(H)$ of bounded linear operators on
some Hilbert space $H$. (The reader can consult \cite{Murphy} as a reference
for the basic theory of C*-algebras.) The group $\mathrm{Aut}\left( A\right) 
$ of automorphisms of $A$ is a Polish group with respect to the topology of
pointwise convergence (see \cite{Raeburn-Rosenberg} page 4). A C*-algebra is
called \textit{unital} if it contains a multiplicative identity, usually
denoted by $1$. If $A$ is unital and $u$ is a unitary element of $A$ (i.e.\
such that $uu^{\ast }=u^{\ast }u=1$), then%
\begin{equation*}
\mathrm{Ad}(u)\left( x\right) =uxu^{\ast }
\end{equation*}%
defines an automorphism \textrm{Ad}$(u)$ of $A$. When $A$ is not unital one
can consider unitary elements of the multiplier algebra of $A$. The\textit{\
multiplier algebra} $M(A)$ of $A$ is the largest unital C*-algebra
containing $A$ as an essential ideal (see \cite{Blackadar} II.7.3). It can
be regarded as the noncommutative analog of the Stone-\v{C}ech
compactification of a locally compact Hausdorff space. The multiplier
algebra of a separable C*-algebra $A$ is not norm separable (unless $A$ is
unital, in which case $M(A)$ coincides with $A$). Nonetheless the \textit{%
strict topology }(see \cite{Blackadar} II.7.3.11) of $M(A)$ is Polish and
induces a Polish group structure on the group $U(A)$ of unitary elements of $%
M(A)$. If $u$ is a unitary multiplier of $A$, i.e.\ a an element of $U(A)$,
then one can define as before the automorphism $\mathrm{Ad}(u)$ of $A$. An
automorphism of $A$ is called \textit{inner }if it is of the form $\mathrm{Ad%
}(u)$ for some unitary multiplier $u$, and \textit{outer }otherwise. Inner
automorphisms of a separable C*-algebra $A$ form a Borel normal subgroup of $%
\mathrm{Aut}(A)$ (see\ \cite{Phillips1} Proposition 2.4). Two automorphisms $%
\alpha $ and $\beta $ of $A$ are called \textit{unitarily equivalent} if $%
\alpha \circ \beta ^{-1}$ is inner or, equivalently,%
\begin{equation*}
\alpha (x)=\beta (uxu^{\ast })
\end{equation*}%
for some unitary multiplier $u$ and every $x\in A$. This defines a Borel
equivalence relation on $\mathrm{Aut}(A)$.

In the rest of the paper, \textit{we assume all C*-algebras to be norm
separable}, apart from multiplier algebras and enveloping von Neumann
algebras. The \textit{enveloping von Neumann algebra} or \textit{second dual}
$A^{\ast \ast }$ of a C*-algebra $A$ (see \cite{Pedersen} 3.7.6 and 3.7.8)
is a von Neumann algebra isometrically isomorphic to the second dual of $A$.
The $\sigma $\textit{-weak topology} on $A^{\ast \ast }$ is the weak*
topology of $A^{\ast \ast }$ regarded as the dual Banach space of $A^{\ast }$%
. The algebra $A$ can be identified with a $\sigma $-weakly dense subalgebra
of $A^{\ast \ast }$. Moreover (see \cite{Pedersen} 3.12.3) we can identify
the multiplier algebra $M\left( A\right) $ of $A$ with the idealizer of $A$
inside $A^{\ast \ast }$, i.e.\ the algebra of elements $x$ such that $xa\in
A $ and $ax\in A$ for every $a\in A$. Analogously, the \textit{unitization} $%
\tilde{A}$ of $A$ (see \cite{Blackadar} II.1.2) is identified with the
subalgebra of $M\left( A\right) $ generated by $A$ and $1$. If $x$ is a 
\textit{normal} element of $A$, i.e.\ commuting with its adjoint, and $f$ is
a complex-valued continuous function defined on the spectrum of $x$, $%
f\left( x\right) $ denotes the element of $\tilde{A}$ obtained from $x$ and $%
f$ using functional calculus (II.2 of \cite{Blackadar} is a complete
reference for the basic notions of spectral theory and continuous functional
calculus in operator algebras). If $x,y$ are element of a C*-algebra, then $%
\left[ x,y\right] $ denotes their \textit{commutator} $xy-yx$; moreover if $%
S $ is a subset of a C*-algebra $A$, then $S^{\prime }\cap A$ denotes the
relative commutant of $S$ in $A$ (see \cite{Blackadar} I.2.5.3). The set $%
\mathbb{N}$ of natural numbers is supposed not to contain $0$. Boldface
letters $\mathbf{t}$ and $\mathbf{s}$ indicate sequences of real numbers
whose $n$-th terms are $t_{n}$ and $s_{n}$ respectively. Analogously $%
\mathbf{x}$ stands for the sequence $(x_{n})_{n\in \mathbb{N}}$ of elements
of a C*-algebra $A$.

\section{Nonclassifiability criteria}

\label{Section: criteria}Recall that a subset $A$ of a Polish space $X$ has
the Baire property (Definition 8.21 of \cite{Kechris}) if its symmetric
difference with some open set is meager. A function between Polish spaces is
Baire measurable (Definition 8.37 of \cite{Kechris}) if the inverse image of
any open set has the Baire property. Observe that, in particular, any Borel
function is Baire measurable. Suppose that $E$ and $R$ are equivalence
relations on Polish spaces $X$ and $Y$ respectively. We say that $E$ is 
\textit{generically }$R$-\textit{ergodic} if, for every Baire measurable
function $f:X\rightarrow Y$ such that $f(x)Rf(y)$ whenever $xEy$, there is a
comeager subset $C$ of $X$ such that $f(x)Rf(y)$ for every $x,y\in C$ (cf.\ 
\cite{Gao} Definition 10.1.4). Observe that if $E$ is generically $R$%
-ergodic and no equivalence class of $E$ is comeager then, in particular, $E$
is not Borel reducible to $R$.

One of the main tools in the study of Borel complexity of equivalence
relations is Hjorth's theory of turbulence. A standard reference for this
subject is \cite{Hjorth-book}. Turbulence is a dynamical property of a
continuous group action of a Polish group $G$ on a Polish space $X$ (see 
\cite{Hjorth-book} Definition 3.13). The main result about turbulent actions
is the following result of Hjorth (Theorem 3.21 in \cite{Hjorth-book}):%
\newline

\textit{The orbit equivalence relation }$E_{G}^{X}$\textit{\ associated with
a turbulent action }$G\curvearrowright X$ \textit{of a Polish group }$G$%
\textit{\ on a Polish space }$X$\textit{\ is generically }$\simeq _{\mathcal{%
C}}$-\textit{ergodic for every class }$\mathcal{C}$\textit{\ of countable
structures, where }$\simeq _{\mathcal{C}}$\textit{\ denotes the relation of
isomorphism for elements of $\mathcal{C}$. Moreover }$E_{G}^{X}$ \textit{has
meager equivalence classes, and hence it is not classifiable by countable
structures.}\newline

This result is valuable not only on its own, but also because it allows one
to prove nonclassification results via the following two lemmas:

\begin{lemma}
\label{Lemma: generically ergodic criterion}Suppose that $E$, $F$, and $R$
are equivalence relations on Polish spaces $X$, $Y$, and $Z$ respectively,
and that $F$ is generically $R$-ergodic. If there is a comeager subset $%
\tilde{C}$ of $Y$ and a Baire measurable function $f:\widetilde{C}%
\rightarrow X$ such that:

\begin{itemize}
\item $f(x)Ef(y)$ for any $x,y\in \widetilde{C}$ such that $xFy$;

\item $f[C]$ is comeager in $X$ for every comeager subset $C$ of $\widetilde{%
C} $;
\end{itemize}

then the relation $E$ is generically $R$-ergodic as well.
\end{lemma}

\begin{proof}
Suppose that $g:X\rightarrow Z$ is a Baire measurable function such that $%
g(x)Rg(x^{\prime })$ for any $x,x^{\prime }\in X$ such that $xEx^{\prime }$.
The composition $g\circ f$ is a Baire measurable function from $\tilde{C}$
to $Z$ such that $(f\circ g)(y)R(f\circ g)(y^{\prime })$ for any $%
y,y^{\prime }\in \tilde{C}$ such that $yFy^{\prime }$. Since $\tilde{C}$ is
comeager in $Y$, and $F$ is generically $R$-ergodic, there is a comeager
subset $C$ of $\tilde{C}$ such that $(g\circ f)(y)R(g\circ f)(y^{\prime })$
for every $y,y^{\prime }\in C$. Therefore, $f[C]$ is a comeager subset of $X$
such that $g(x)Rg(x^{\prime })$ for every $x,x^{\prime }\in f[C]$. 
\end{proof}

Observe that if $f$ is continuous, open, and onto, then it will
automatically satisfy the second condition of Lemma \ref{Lemma: generically
ergodic criterion}.

\begin{lemma}
\label{Lemma: non-classifiability criterion}Suppose that $E$ and $F$ are
equivalence relations on Polish spaces $X$ and $Y$ respectively, and $F$ is
generically $\simeq _{\mathcal{C}}$-ergodic for every class $\mathcal{C}$ of
countable structures. If there is a Baire measurable function $%
f:Y\rightarrow X$ such that:

\begin{itemize}
\item $f(x)Ef(y)$ whenever $xFy$;

\item for every comeager subset $C$ of $Y$ there are $x,y\in C$ such that $%
f(x)\not{\!\!E}f(y)$;
\end{itemize}

then the relation $E$ is not classifiable by countable structures.
\end{lemma}

\begin{proof}
Suppose by contradiction that there is a class $\mathcal{C}$ of countable
structures and a Borel reduction $g:X\rightarrow \mathcal{C}$ of $E$ to $%
\simeq _{\mathcal{C}}$. The composition $g\circ f:Y\rightarrow \mathcal{C}$
is a Baire measurable function from $Y$ to $\mathcal{C}$ such that $(g\circ
f)(y)\simeq _{\mathcal{C}}(g\circ f)(y^{\prime })$ for any $y,y^{\prime }\in
Y$ such that $yFy^{\prime }$. Since $F$ is generically $\simeq _{\mathcal{C}%
} $-ergodic, there is a comeager subset $C$ of $Y$ such that $(g\circ
f)(y)\simeq _{\mathcal{C}}(g\circ f)(y^{\prime })$ for every $y,y^{\prime
}\in C$. Therefore, being $g$ a reduction of $E$ to $\simeq _{\mathcal{C}}$, 
$f\left( y\right) Ef\left( y^{\prime }\right) $ for every $y,y^{\prime }\in
C $. This contradicts our assumptions.
\end{proof}

Consider $\mathbb{R}^{\mathbb{N}}$ as a Polish space with the product
topology and $\ell ^{1}$ as a Polish group with its Banach space topology.
The fact that the action of $\ell ^{1}$ on $\mathbb{R}^{\mathbb{N}}$ by
translation is turbulent is a particular case of Proposition 3.25 in \cite%
{Hjorth-book}. It then follows by Hjorth's turbulence theorem that the
associated orbit equivalence relation $E_{\mathbb{R}^{\mathbb{N}}}^{\ell
^{1}}$ is generically $\simeq _{\mathcal{C}}$-ergodic for every class $%
\mathcal{C}$ of countable structures. It is not difficult to see that the
function $f:\left( \mathbb{R}\backslash \left\{ 0\right\} \right) ^{\mathbb{N%
}}\rightarrow \left( 0,1\right) ^{\mathbb{N}}$ defined by%
\begin{equation*}
f\left( \mathbf{t}\right) =\left( \frac{\left\vert t_{n}\right\vert }{%
\left\vert t_{n}\right\vert +1}\right) _{n\in \mathbb{N}}
\end{equation*}%
satisfies both the first (being continuous, open, and onto) and the second
condition of Lemma \ref{Lemma: generically ergodic criterion}, where:

\begin{itemize}
\item $F$ is the relation $E_{\mathbb{R}^{\mathbb{N}}}^{\ell ^{1}}$ of
equivalence modulo $\ell ^{1}$ of sequences of real numbers;

\item $E$ is the relation $E_{\left( 0,1\right) ^{\mathbb{N}}}^{\ell ^{1}}$
of equivalence modulo $\ell ^{1}$ of sequences of real numbers between $0$
and $1$.
\end{itemize}

It follows that the latter relation is generically $\simeq _{\mathcal{C}}$%
-ergodic for every class $\mathcal{C}$ of countable structures. Considering
the particular case of Lemma \ref{Lemma: non-classifiability criterion} when 
$F$ is the relation $E_{\left( 0,1\right) ^{\mathbb{N}}}^{\ell ^{1}}$ one
obtains the following nonclassifiability criterion:

\begin{criterion}
\label{Criterion: non-classifiability}If $E$ is an equivalence relation on a
Polish space $X$ and there is a Baire measurable function $f:\left(
0,1\right) ^{\mathbb{N}}\rightarrow X$ such that:

\begin{itemize}
\item $f(\mathbf{x})Ef(\mathbf{y})$ for any $\mathbf{x},\mathbf{y}\in \left(
0,1\right) ^{\mathbb{N}}$ such that $\mathbf{x}-\mathbf{y}\in \ell ^{1}$;

\item any comeager subset of $\left( 0,1\right) ^{\mathbb{N}}$ contains
elements $\mathbf{x},\mathbf{y}$ such that $f(\mathbf{x})\not{\!\!E}f(%
\mathbf{y})$;
\end{itemize}

then the relation $E$ is not classifiable by countable structures.
\end{criterion}

In order to apply Criterion \ref{Criterion: non-classifiability} we will
need the following fact about nonmeager subsets of $\left( 0,1\right) ^{%
\mathbb{N}}$:

\begin{lemma}
\label{Lemma: uncountable}If $X$ is a nonmeager subset of $\left( 0,1\right)
^{\mathbb{N}}$, then there is an uncountable $Y\subset X$ such that, for
every pair of distinct points $\mathbf{s},\mathbf{t}$ of $Y$, $\left\Vert 
\mathbf{s}-\mathbf{t}\right\Vert _{\infty }\geq \frac{1}{4}$, where%
\begin{equation*}
\left\Vert \mathbf{s}-\mathbf{t}\right\Vert _{\infty }=\sup_{n\in \mathbb{N}%
}\left\vert t_{n}-s_{n}\right\vert \text{.}
\end{equation*}
\end{lemma}

\begin{proof}
Define for every $\mathbf{s}\in \left( 0,1\right) $,%
\begin{equation*}
K_{\mathbf{s}}=\left\{ \mathbf{t}\in \left( 0,1\right) ^{\mathbb{N}}\mid
\left\Vert \mathbf{t}-\mathbf{s}\right\Vert \leq \frac{1}{4}\right\} \text{.}
\end{equation*}%
Observe that $K_{\mathbf{s}}$ is a closed nowhere dense subset of $\left(
0,1\right) ^{\mathbb{N}}$. Consider the class $\mathcal{A}$ of subsets $Y$
of $X$ with the property that, for every $\mathbf{s,t}$ in $Y$ distinct, $%
\left\Vert \mathbf{s}-\mathbf{t}\right\Vert \geq \frac{1}{4}$. If $\mathcal{A%
}$ is partially ordered by inclusion, then it has some maximal element $Y$
by Zorn's lemma. By maximality,%
\begin{equation*}
X\subset \bigcup_{\mathbf{t}\in Y}\left\{ \mathbf{s}\in \left( 0,1\right) ^{%
\mathbb{N}}\mid \left\Vert \mathbf{t}-\mathbf{s}\right\Vert _{\infty }\leq 
\frac{1}{4}\right\} \text{.}
\end{equation*}%
Being $X$ nonmeager, $Y$ is uncountable.
\end{proof}

\section{The case of algebras with outer derivations}

\label{Section: outer} The aim of this section is to show that if a
C*-algebra $A$ has an outer derivation, then the relation of unitary
equivalence of approximately inner automorphisms of $A$ is not classifiable
by countable structures. In proving this fact we will also show that any
such C*-algebra satisfies a seemingly stronger property, that we will refer
to as Property AEP (see Definition \ref{Definition: property AEP}).

A \textit{derivation} of a C*-algebra $A$ is a linear function%
\begin{equation*}
\delta :A\rightarrow A
\end{equation*}%
satisfying the \textit{derivation identity}:%
\begin{equation*}
\delta (xy)=\delta (x)y+x\delta (y)
\end{equation*}%
for $x,y\in A$. The derivation identity implies that $\delta $ is a bounded
linear operator on $A$ (see \cite{Pedersen} Proposition 8.6.3). The set $%
\Delta (A)$ of derivations of $A$ is a closed subspace of the Banach space $%
B(A)$ of bounded linear operators on $A$. A derivation is called a
*-derivation if it is a positive linear operator, i.e.\ it sends positive
elements to positive elements. Any element $a$ of the multiplier algebra of $%
A$ defines a derivation $\mathrm{ad}(ia)$ of $A$, by%
\begin{equation*}
\mathrm{ad}(ia)\left( x\right) =\left[ ia,x\right] \text{.}
\end{equation*}%
This is a *-derivation if and only if $a$ is self-adjoint. A derivation of
this form is called \textit{inner}, and \textit{outer }otherwise. More
generally, if $a$ is an element of the enveloping von Neumann algebra of $A$
that \textit{derives} $A$, i.e.\ $ax-xa\in A$ for any $x\in A$, then one can
define the (not necessarily inner) derivation $\mathrm{ad}(ia)$ of $A$.
Since any derivation is linear combination of *-derivations (see \cite%
{Pedersen} 8.6.2), the existence of an outer derivation is equivalent to the
existence of an outer *-derivation. The set $\Delta _{0}(A)$ of inner
derivations of $A$ is a Borel (not necessarily closed) subspace of $\Delta
(A)$. The norm on $\Delta _{0}(A)$ defined by%
\begin{equation*}
\left\Vert \mathrm{ad}(ia)\right\Vert _{\Delta _{0}(A)}=\inf \left\{
\left\Vert a-z\right\Vert \mid z\in A^{\prime }\cap A\right\}
\end{equation*}%
makes $\Delta _{0}(A)$ a separable Banach space isometrically isomorphic to
the quotient of $A$ by its center $A^{\prime }\cap A$. The inclusion of $%
\Delta _{0}(A)$ in $\Delta (A)$ is continuous, and the closure $\overline{%
\Delta _{0}(A)}$ of $\Delta _{0}(A)$ in $\Delta (A)$ is a closed separable
subspace of $\Delta (A)$. If $\delta $ is a *-derivation then the
exponential $\mathrm{exp}(\delta )$ of $\delta $, regarded as an element of
the Banach algebra $B(A)$ of bounded linear operators of $A$, is an
automorphism of $A$. Automorphisms of this form are called derivable. If $%
\delta =\mathrm{ad}(ia)$ is inner then%
\begin{equation*}
\mathrm{exp}(\delta )=\mathrm{Ad}({\mathrm{exp}}({ia))}
\end{equation*}%
is inner as well. Lemma \ref{Lemma: outer aut implies outer d} provides a
partial converse to this statement. (The converse is in fact false in
general.) For more information on derivations and derivable automorphisms,
the reader is referred to \cite{Pedersen} Section 8.6.

Recall that an irreducible representation of a C*-algebra $A$ is a
*-homomorphism $\pi :A\rightarrow B\left( H\right) $, where $H$ is a
(necessarily separable) Hilbert space, such that no nontrivial proper
subspace of $H$ is invariant for $\pi \left[ A\right] $ (see \cite{Blackadar}
Definition II.6.1.1 and Proposition II.6.1.8). Two irreducible
representations $\pi _{0},\pi _{1}$ of $A$ are \textit{unitarily equivalent}
if there is a unitary element $u$ of $B\left( H\right) $ such that $\pi
_{1}(x)=u\pi _{0}(x)u^{\ast }$ for every $x\in A$. An ideal of a C*-algebra $%
A$ is called \textit{primitive} if it is the kernel of an irreducible
representation of $A$. A C*-algebra $A$ is called primitive if $\left\{
0\right\} $ is a primitive ideal in $A$, i.e.\ $A$ has a faithful
irreducible representation. The \textit{primitive spectrum }$\check{A}$ of $%
A $ is the space of primitive ideals of $A$ endowed with the hull-kernel
topology described in \cite{Pedersen} 4.1.3. The \textit{spectrum} $\hat{A}$
of $A$ is the space of unitary equivalence classes of irreducible
representations of $A$ endowed with the Jacobson topology defined in \cite%
{Pedersen} 4.1.12.

\begin{lemma}
\label{Lemma: outer aut implies outer d}Suppose that $A$ is a primitive
C*-algebra. If $\delta $ is a *-derivation of $A$ with operator norm
strictly smaller than $2\pi $ such that $\mathrm{exp}(\delta )$ is inner,
then $\delta $ is inner.
\end{lemma}

The lemma is proved in \cite{Kadison-Lance-Ringrose} (Theorem 4.6 and Remark
4.7) under the additional assumption that $A$ is unital. It is not difficult
to check that the same proof works without change in the nonunital case.

\begin{definition}
\label{Definition: E_x}Suppose that $A$ is a C*-algebra, $\left(
a_{n}\right) _{n\in \mathbb{N}}$ is a dense sequence in the unit ball of $A$%
, and $\mathbf{x}=(x_{n})_{n\in \mathbb{N}}$ is a sequence of pairwise
orthogonal positive contractions of $A$ such that for every $n\in \mathbb{N}$
and $i\leq n$,%
\begin{equation*}
\left\Vert \lbrack x_{n},a_{i}]\right\Vert \leq 2^{-n}\text{.}
\end{equation*}%
Since the $x_{n}$'s are pairwise orthogonal, if $\mathbf{t}$ is a sequence
of real numbers of absolute value at most $1$, then the series%
\begin{equation*}
\sum_{n\in \mathbb{N}}t_{n}x_{n}
\end{equation*}%
converges to a self-adjoint element of $A^{\ast \ast }$. Moreover, the
sequence of inner automorphisms%
\begin{equation*}
\left( \mathrm{Ad}\left( \mathrm{exp}\left( i\sum_{k\leq n}t_{k}x_{k}\right)
\right) \right) _{n\in \mathbb{N}}
\end{equation*}%
of $A$ converges to the approximately inner automorphism%
\begin{equation*}
\alpha _{\mathbf{t}}:=\mathrm{Ad}{\left( \mathrm{\mathrm{exp}}\left(
i\sum_{n\in \mathbb{N}}t_{n}x_{n}\right) \right) }\text{.}
\end{equation*}%
The equivalence relation $E_{\mathbf{x}}$ on $\left( 0,1\right) ^{\mathbb{N}%
} $ is defined by%
\begin{equation*}
\mathbf{s}E_{\mathbf{x}}\mathbf{t\quad }\text{iff}\mathbf{\quad }\alpha _{%
\mathbf{t}}\text{ and }\alpha _{\mathbf{s}}\text{ are unitarily equivalent.}
\end{equation*}
\end{definition}

Observe that this equivalence relation is finer than the relation of $\ell
^{1}$-equivalence introduced in Section \ref{Section: criteria}. In fact if $%
\mathbf{s},\mathbf{t}\in \left( 0,1\right) ^{\mathbb{N}}$ and $\mathbf{s}-%
\mathbf{t}\in \ell ^{1}$, then the series%
\begin{equation*}
\sum_{n\in \mathbb{N}}(t_{n}-s_{n})x_{n}
\end{equation*}%
converges in $A$. It is then easily verified that%
\begin{equation*}
u:=\mathrm{exp}\left( i\sum_{n\in \mathbb{N}}\left( t_{n}-s_{n}\right)
x_{n}\right)
\end{equation*}%
is a unitary multiplier of $A$ such that%
\begin{equation*}
\mathrm{Ad}(u)\circ \alpha _{\mathbf{s}}=\alpha _{\mathbf{t}}\text{.}
\end{equation*}%
Therefore, if the equivalence classes of $E_{\mathbf{x}}$ are meager, the
continuous function

\begin{eqnarray*}
\left( 0,1\right) ^{\mathbb{N}} &\rightarrow &\mathrm{Aut}\left( A\right) \\
\mathbf{t} &\mapsto &\alpha _{\mathbf{t}}
\end{eqnarray*}%
satisfies the hypothesis of Criterion \ref{Criterion: non-classifiability}.
This concludes the proof of the following lemma:

\begin{lemma}
\label{Lemma: fundamental outer}Suppose that $A$ is a C*-algebra. If for
some sequence $\mathbf{x}$ of pairwise orthogonal positive contractions of $%
A $ the equivalence relation $E_{\mathbf{x}}$ has meager equivalence
classes, then the approximately inner automorphisms of $A$ are not
classifiable by countable structures.
\end{lemma}

Lemma \ref{Lemma: fundamental outer} motivates the following definition.

\begin{definition}
\label{Definition: property AEP}A C*-algebra $A$ has Property AEP if for
every dense sequence $\left( a_{n}\right) _{n\in \mathbb{N}}$ in the unit
ball of $A$ there is a sequence $\mathbf{x}=\left( x_{n}\right) _{n\in 
\mathbb{N}}$ of pairwise orthogonal positive contractions of $A$ such that:

\begin{enumerate}
\item \label{Itemize: 1}$\left\Vert \left[ x_{n},a_{i}\right] \right\Vert
<2^{-n}$ for $i\in \left\{ 1,2,\ldots ,n\right\} $;

\item \label{Itemize: 2}the relation $E_{\mathbf{x}}$ as in Definition \ref%
{Definition: E_x} has meager conjugacy classes.
\end{enumerate}
\end{definition}

It is clear that if a C*-algebra $A$ has Property AEP, then $A$ has an
outer\linebreak *-derivation. In fact, if $\mathbf{s},\mathbf{t}\in \left(
0,1\right) ^{\mathbb{N}}$ are such that $\mathbf{s} \not{\!\!\!E}_{\mathbf{x}%
}\ \mathbf{t}$, then the self-adjoint element%
\begin{equation*}
a=\sum_{n\in \mathbb{N}}\left( t_{n}-s_{n}\right) x_{n}
\end{equation*}%
of $A^{\ast \ast }$ derives $A$. The automorphism \textrm{Ad}$(\mathrm{%
\mathrm{exp}}(ia))$ is outer, and hence such is the *-derivation $\mathrm{ad}%
(ia)$. The rest of this section is devoted to prove that, conversely, if $A$
has an outer *-derivation, then $A$ has Property AEP.

The following lemma shows that primitive nonsimple C*-algebras have Property
AEP. The main ingredients of the proof are borrowed from Lemma 6 of \cite%
{Elliott3} and Lemma 3.2 of \cite{Akemann-Pedersen}.

\begin{lemma}
\label{Lemma: primitive}If $A$ is a primitive nonsimple C*-algebra, then it
has Property AEP.
\end{lemma}

\begin{proof}
Fix a faithful irreducible representation $\pi :A\rightarrow B(H)$. By
Theorem 3.7.7 of \cite{Pedersen} $\pi $ extends to a $\sigma $-weakly
continuous representation $\pi ^{\ast \ast }:A^{\ast \ast }\rightarrow B(H)$%
. Fix a dense sequence $\left( a_{n}\right) _{n\in \mathbb{N}}$ in the unit
ball of $A$ and a strictly positive contraction $b_{0}$ of $A$ (see
Proposition II.4.2.1 of \cite{Blackadar} for a characterization of strictly
positive elements). As in the proof of Lemma 3.2 in \cite{Akemann-Pedersen},
one can define a sequence $\mathbf{x}=\left( x_{n}\right) _{n\in \mathbb{N}}$
of pairwise orthogonal projections such that for some $\varepsilon >0$ and
every $k,n\in \mathbb{N}$ such that $k\leq n$,

\begin{itemize}
\item $\left\Vert x_{n}b_{0}\right\Vert >\varepsilon $;

\item $\left\Vert \left[ x_{n},a_{k}\right] \right\Vert <2^{-n}$.
\end{itemize}

Now suppose by contradiction that the equivalence relation $E_{\mathbf{x}}$
has a nonmeager equivalence class $X$. Thus for every $\mathbf{t},\mathbf{s}%
\in X$ the automorphism%
\begin{equation*}
\alpha _{\mathbf{t},\mathbf{s}}=\mathrm{Ad}{\left( \mathrm{exp}\left(
i\sum_{n\in \mathbb{N}}(t_{n}-s_{n})x_{n}\right) \right) }
\end{equation*}%
is inner. Fix $\mathbf{s},\mathbf{t}\in X$. Observe that $\alpha _{\mathbf{t}%
,\mathbf{s}}$ is the exponential of the *-derivation%
\begin{equation*}
\delta _{\mathbf{t},\mathbf{s}}=\mathrm{ad}{\left( i\sum_{n\in \mathbb{N}%
}(t_{n}-s_{n})x_{n}\right) }\text{.}
\end{equation*}%
By Lemma \ref{Lemma: outer aut implies outer d} the *-derivation $\delta _{%
\mathbf{t},\mathbf{s}}$ is inner. Thus, there is an element $z_{\mathbf{t},%
\mathbf{s}}$ of the center of the enveloping von Neumann algebra of $A$ such
that%
\begin{equation*}
\sum_{n\in \mathbb{N}}(t_{n}-s_{n})x_{n}+z_{\mathbf{t},\mathbf{s}}\in M(A)%
\text{.}
\end{equation*}%
The image of a central element of $A^{\ast \ast }$ under $\pi $ belongs to
the relative commutant of $\pi \lbrack A]$ in $B(H)$, which consists only of
scalar multiples of the identity by II.6.1.8 of \cite{Blackadar}. Thus,%
\begin{equation*}
\pi \left( \sum_{n\in \mathbb{N}}(t_{n}-s_{n})x_{n}\right) \in \pi ^{\ast
\ast }\left[ M\left( A\right) \right]
\end{equation*}%
Hence%
\begin{equation*}
\pi \left( b_{0}\sum_{n\in \mathbb{N}}(t_{n}-s_{n})x_{n}\right) \in \pi
\lbrack A]\text{.}
\end{equation*}%
By Lemma \ref{Lemma: uncountable} one can find an uncountable subset $Y$ of $%
X$ such that any pair of distinct elements of $Y$ has uniform distance at
least $\frac{1}{4}$. Fix $\mathbf{s}\in Y$. For all $\mathbf{t},\mathbf{t}%
^{\prime }\in Y$, there is $m\in \mathbb{N}$ such that%
\begin{equation*}
\left\vert t_{m}-t_{m}^{\prime }\right\vert \geq \frac{1}{4}\text{.}
\end{equation*}%
Henceforth,%
\begin{eqnarray*}
&&\left\Vert \pi \left( b_{0}\left( \sum_{n\in \mathbb{N}}(t_{n}-s_{n})x_{n}%
\right) \right) -\pi \left( b_{0}\left( \sum_{n\in \mathbb{N}}(t_{n}^{\prime
}-s_{n})x_{n}\right) \right) \right\Vert \\
&=&\left\Vert \pi \left( b_{0}\sum_{n\in \mathbb{N}}(t_{n}-t_{n}^{\prime
})x_{n}\right) \right\Vert \\
&=&\left\Vert b_{0}\sum_{n\in \mathbb{N}}(t_{n}-t_{n}^{\prime
})x_{n}\right\Vert \\
&\geq &\left\Vert b_{0}\sum_{n\in \mathbb{N}}(t_{n}-t_{n}^{\prime
})x_{n}x_{m}a_{0}\right\Vert \\
&\geq &\left\vert t_{m}-t_{m}^{\prime }\right\vert \left\Vert
(x_{m}b_{0})^{\ast }(x_{m}b_{0})\right\Vert \geq \frac{\varepsilon ^{2}}{4}%
\text{.}
\end{eqnarray*}%
Since $Y$ is uncountable this contradicts the separability of $\pi \left[ A%
\right] $.
\end{proof}

In order to prove Property AEP for all C*-algebra with outer *-derivations
we need the fact that Property AEP is \textit{liftable}. This means that if
a *-homomorphic image of a C*-algebra $A$ has Property AEP, then $A$ has
Property AEP. (For an exhaustive introduction to liftable properties the
reader is referred to Chapter 8 of \cite{Loring}.)

\begin{lemma}
\label{Lemma: lifting}If $\pi :A\rightarrow B$ is a surjective
*-homomorphism and $B$ has Property AEP, then $A$ has Property AEP.
\end{lemma}

\begin{proof}
Suppose that $\left( a_{n}\right) _{n\in \mathbb{N}}$ is a dense sequence in 
$A$. Thus, $\left( \pi (a_{n})\right) _{n\in \mathbb{N}}$ is a dense
sequence in $B$. Pick a sequence $\left( y_{n}\right) _{n\in \mathbb{N}}$ in 
$B$ obtained from $\left( \pi (a_{n})\right) _{n\in \mathbb{N}}$ as in the
definition of Property AEP. By Lemma 10.1.12 of \cite{Loring}, there is a
sequence $\left( z_{n}\right) _{n\in \mathbb{N}}$ of pairwise orthogonal
positive contractions of $A$ such that $\pi (z_{n})=y_{n}$ for every $n\in 
\mathbb{N}$. Fix an increasing quasicentral approximate unit of $\mathrm{Ker}%
(\pi )$ (cf.\ \cite{Blackadar} Section II.4.3), i.e.\ a sequence $\left(
u_{k}\right) _{k\in \mathbb{N}}$ of elements of $\mathrm{Ker}\left( \pi
\right) $ such that:

\begin{itemize}
\item $\lim_{k\rightarrow +\infty }\left\Vert u_{k}x-x\right\Vert
=\lim_{k\rightarrow +\infty }\left\Vert xu_{k}-x\right\Vert =0$ for every $%
x\in \mathrm{Ker}\left( \pi \right) $;

\item $\lim_{k\rightarrow +\infty }\left\Vert \left[ u_{k},a\right]
\right\Vert =0$ for every $a\in A$.
\end{itemize}

For every $n,i\in \mathbb{N}$ such that $i\leq n$, by Proposition II.5.1.1
of \cite{Blackadar},%
\begin{eqnarray*}
\lim_{k\rightarrow +\infty }\left\Vert z_{n}^{\frac{1}{2}}(1-u_{k})z_{n}^{%
\frac{1}{2}}a_{i}-a_{i}z_{n}^{\frac{1}{2}}(1-u_{k})z_{n}^{\frac{1}{2}%
}\right\Vert &=&\lim_{k\rightarrow +\infty }\left\Vert
(1-u_{k})(z_{n}a_{i}-a_{i}z_{n})\right\Vert \\
&=&\left\Vert y_{n}\,\pi (a_{i})-\pi (a_{i})\,y_{n}\right\Vert <2^{-n}\text{.%
}
\end{eqnarray*}%
Thus, there is $k_{n}\in \mathbb{N}$ such that, if%
\begin{equation*}
x_{n}=z_{n}^{\frac{1}{2}}(1-u_{k_{n}})z_{n}^{\frac{1}{2}}\text{,}
\end{equation*}%
then%
\begin{equation*}
\left\Vert x_{n}a_{i}-a_{i}x_{n}\right\Vert <2^{-n}
\end{equation*}%
for every $i\leq n$. Observe that $\left( x_{n}\right) _{n\in \mathbb{N}}$
is a sequence of pairwise orthogonal positive contractions of $A$. Moreover,
if $E\subset \left( 0,1\right) ^{\mathbb{N}}$ is nonmeager, consider $%
\mathbf{s},\mathbf{t}\in E$ such that the automorphism%
\begin{equation*}
\mathrm{Ad}{\left( \mathrm{exp}\left( i\sum_{n\in \mathbb{N}%
}(t_{n}-s_{n})y_{n}\right) \right) }
\end{equation*}%
of $B$ is outer. We claim that the automorphism%
\begin{equation*}
\mathrm{Ad}{\left( \mathrm{exp}\left( i\sum_{n\in \mathbb{N}%
}(t_{n}-s_{n})x_{n}\right) \right) }
\end{equation*}%
of $A$ is outer. Suppose that this is not the case. Thus, there is $z$ in
the center of the enveloping von Neumann algebra of $A$ such that%
\begin{equation*}
\mathrm{exp}\left( i\sum_{n\in \mathbb{N}}(t_{n}-s_{n})x_{n}\right) +z\in
U(A)\text{.}
\end{equation*}%
Denoting by $\pi ^{\ast \ast }:A^{\ast \ast }\rightarrow B^{\ast \ast }$ the
normal extension of $\pi $ (see III.5.2.10 of \cite{Blackadar}), one has that%
\begin{equation*}
\mathrm{exp}\left( i\sum_{n\in \mathbb{N}}(t_{n}-s_{n})y_{n}\right) +\pi
^{\ast \ast }\left( z\right) =\pi ^{\ast \ast }\left( \mathrm{exp}\left(
i\sum_{n\in \mathbb{N}}(t_{n}-s_{n})x_{n}\right) +z\right) \in U(B)
\end{equation*}%
by Theorem 4.2 of \cite{Akemann-Pedersen-Tomiyama}. Since $\pi ^{\ast \ast
}\left( z\right) $ belongs to the center of the enveloping von Neumann
algebra of $B$,%
\begin{equation*}
\mathrm{exp}\left( i\sum_{n\in \mathbb{N}}(t_{n}-s_{n})y_{n}\right) +\pi
^{\ast \ast }\left( z\right)
\end{equation*}%
is a unitary multiplier of $B$ that implements%
\begin{equation*}
\mathrm{Ad}{\left( \mathrm{exp}\left( i\sum_{n\in \mathbb{N}%
}(t_{n}-s_{n})y_{n}\right) \right) }\text{.}
\end{equation*}%
Hence, the latter automorphism of $B$ is inner, contradicting the assumption.
\end{proof}

Liftability of Property AEP allows one to easily bootstrap Property AEP from
primitive nonsimple C*-algebras to C*-algebra whose primitive spectrum is
not $T_{1}$.

\begin{lemma}
\label{Lemma: non t1}If $A$ is a C*-algebra whose primitive spectrum $\check{%
A}$ is not $T_{1}$, then $A$ has Property AEP.
\end{lemma}

\begin{proof}
Since $\check{A}$ is not $T_{1}$, by 4.1.4 of \cite{Pedersen} there is an
irreducible representation $\pi $ of $A$ whose kernel is not a maximal
ideal. This implies that the image of $A$ under $\pi $ is a nonsimple
primitive C*-algebra. By Lemma \ref{Lemma: primitive} the latter C*-algebra
has Property AEP. Therefore, being Property AEP liftable by Lemma \ref%
{Lemma: lifting}, $A$ has Property AEP.
\end{proof}

In order to show that a C*-algebra $A$ has Property AEP, it is sometimes
easier to show that it has a stronger property that we will refer to as
Property AEP$^{+}$. Property AEP$^{+}$ appears, without being explicitly
defined, in the proofs of Lemma 17, Theorem 18, and the main Theorem of \cite%
{Elliott3}, as well as in the proofs of Lemma 3.5 and Lemma 3.6 of \cite%
{Akemann-Pedersen}.

Recall that a bounded sequence $\left( x_{n}\right) _{n\in \mathbb{N}}$ of
elements of $A$ is called central if for every $a\in A$,%
\begin{equation*}
\lim_{n\rightarrow +\infty }\left\Vert \left[ x_{n},a\right] \right\Vert =0%
\text{.}
\end{equation*}%
The beginning of Section \ref{Section: inner} contains a discussion about
the notion of central sequence, the related notion of hypercentral sequence,
and their basic properties.

\begin{definition}
\label{Definition: AEP+}A C*-algebra $A$ has Property AEP$^{+}$ if there is
a sequence $\left( \pi _{n}\right) _{n\in \mathbb{N}}$ of irreducible
representations of $A$ such that, for some positive contraction $b_{0}$ of $%
A $ and a central sequence $\left( x_{n}\right) _{n\in \mathbb{N}}$ of
pairwise orthogonal positive contractions of $A$:

\begin{itemize}
\item the sequence%
\begin{equation*}
\left( \pi _{n}((x_{n}-\lambda )b_{0})\right) _{n\in \mathbb{N}}
\end{equation*}%
does not converge to $0$ for any $\lambda \in \mathbb{C}$;

\item $x_{m}\in \mathrm{Ker}(\pi _{n})$ for every pair of distinct natural
numbers $n,m$.
\end{itemize}
\end{definition}

To prove that Property AEP$^{+}$ is stronger than Property AEP we will need
the following lemma:

\begin{lemma}
\label{Lemma: exponentials}Fix a strictly positive real number $\eta $. For
every $\varepsilon >0$ there is $\delta >0$ such that for every C*-algebra $%
A $ and every pair of positive contractions $x,b$ of $A$ such that $%
\left\Vert a\right\Vert \geq \eta $, if 
\begin{equation*}
\left\Vert (\mathrm{exp}(ix)-\mu )b\right\Vert \leq \delta
\end{equation*}%
for some $\mu \in \mathbb{C}$ then%
\begin{equation*}
\left\Vert (x-\lambda )b\right\Vert \leq \varepsilon
\end{equation*}%
for some $\lambda \in \mathbb{C}$.
\end{lemma}

\begin{proof}
Fix $\varepsilon >0$. Pick a polynomial%
\begin{equation*}
p(Z)=\rho _{0}+\rho _{1}Z+\ldots +\rho _{n}Z^{n}
\end{equation*}%
such that%
\begin{equation*}
\left\vert p(\mathrm{exp}(it))-t\right\vert \leq \frac{\varepsilon }{2}
\end{equation*}%
for every $t\in \left[ 0,1\right] $. If $\mu \in \mathbb{C}$ is such that $%
\left\vert \mu \right\vert \leq \frac{2}{\eta }$, define $p_{\mu }(Z)$ to be
the polynomial in $Z$ obtained by $p(Z)$ replacing the indeterminate $Z$ by $%
Z+\mu $. Observe that the $j$-th coefficient of $p_{\mu }(Z)$ is 
\begin{equation*}
\rho _{j}^{\mu }=\sum_{i=j}^{n}\rho _{i}\binom{i}{j}\mu ^{j-i}
\end{equation*}%
for $0\leq j\leq n$. Finally define 
\begin{equation*}
C=\sum_{1\leq j\leq i\leq n}\left\vert \rho _{i}\right\vert \binom{i}{j}%
\left( \frac{3}{\eta }\right) ^{j-1}\left( \frac{2}{\eta }\right) ^{j-i}
\end{equation*}%
and%
\begin{equation*}
\delta =\mathrm{min}\left\{ \frac{\varepsilon }{2C},1\right\} \text{.}
\end{equation*}%
Suppose that $A$ is a C*-algebra and $x,b\in A$ are positive contractions
such that $\left\Vert a\right\Vert \geq \eta $ and, for some $\mu \in 
\mathbb{C}$,%
\begin{equation*}
\left\Vert (\mathrm{exp}(ix)-\mu )b\right\Vert \leq \delta \text{.}
\end{equation*}%
Thus,%
\begin{equation*}
\left\vert \mu \right\vert \leq \frac{2}{\eta }\text{.}
\end{equation*}%
Moreover%
\begin{eqnarray*}
\left\Vert (x-\rho _{0}^{\mu })b\right\Vert &=&\left\Vert (p(\mathrm{exp}%
(ix))-\rho _{0}^{\mu })b\right\Vert +\frac{\varepsilon }{2} \\
&=&\left\Vert \left( \sum_{j=1}^{n}\rho _{j}^{\mu }(\mathrm{exp}(ix)-\mu
)^{j}\right) b\right\Vert +\frac{\varepsilon }{2} \\
&\leq &\sum_{j=1}^{n}\left\vert \rho _{j}^{\mu }\right\vert \left\Vert 
\mathrm{exp}\left( ix\right) -\mu \right\Vert ^{j-1}\delta +\frac{%
\varepsilon }{2} \\
&\leq &\sum_{j=1}^{n}\sum_{i=j}^{n}\left\vert \rho _{i}\right\vert \binom{i}{%
j}\left( \frac{2}{\eta }\right) ^{j-i}\left( \frac{3}{\eta }\right)
^{j-1}\delta +\frac{\varepsilon }{2} \\
&\leq &C\delta +\frac{\varepsilon }{2}\leq \varepsilon \text{.}
\end{eqnarray*}%
This concludes the proof.
\end{proof}

We can now prove that Property AEP$^{+}$ is stronger than property AEP.

\begin{proposition}
\label{Proposition: stronger}If a C*-algebra $A$ has Property AEP$^{+}$,
then it has property AEP.
\end{proposition}

\begin{proof}
Suppose that $\left( \pi _{n}\right) _{n\in \mathbb{N}}$ is a sequence of
irreducible representations of $A$, $b_{0}$ is a positive contraction of $A$
of norm $1$, and $\left( x_{n}\right) _{n\in \mathbb{N}}$ is a sequence of
orthogonal positive elements of $A$ as in the definition of Property AEP$%
^{+} $. Fix a dense sequence $\left( a_{n}\right) _{n\in \mathbb{N}}$ in the
unit ball of $A$. After passing to a subsequence of the sequence $\left(
x_{n}\right) _{n\in \mathbb{N}}$, we can assume that for some $\delta >0$,
for every $\lambda \in \mathbb{C}$ and every $n\in \mathbb{N}$,%
\begin{equation*}
\left\Vert \pi _{n}((x_{n}-\lambda )b_{0})\right\Vert \geq \delta
\end{equation*}%
and%
\begin{equation*}
\left\Vert \left[ x_{n},a_{i}\right] \right\Vert <2^{-n}
\end{equation*}%
for $i\leq n$. Thus, for every $\lambda \in \mathbb{C}$, $n\in \mathbb{N}$,
and $t\in \left( \frac{1}{4},1\right) $,%
\begin{equation*}
\left\Vert \pi _{n}((tx_{n}-\lambda )b_{0})\right\Vert \geq \frac{\delta }{4}%
\text{.}
\end{equation*}%
Observe that, in particular,%
\begin{equation*}
\left\Vert \pi _{n}(b_{0})\right\Vert \geq \delta
\end{equation*}%
for every $n\in \mathbb{N}$. By Lemma \ref{Lemma: exponentials}, this
implies that for some $\varepsilon >0$, for every $t\in \left( \frac{1}{4}%
,1\right) $, $n\in \mathbb{N}$ and $\mu \in \mathbb{C}$,%
\begin{equation*}
\left\Vert \pi _{n}((\mathrm{exp}\left( itx_{n}\right) -\mu
)b_{0})\right\Vert \geq \varepsilon \text{.}
\end{equation*}%
Assume by contradiction that there is a nonmeager subset $X$ of $\left(
0,1\right) ^{\mathbb{N}}$ such that for every $\mathbf{s},\mathbf{t}\in X$,
the automorphism%
\begin{equation*}
\mathrm{Ad}{\left( \mathrm{exp}\left( i\sum_{n\in \mathbb{N}%
}(t_{n}-s_{n})x_{n}\right) \right) }
\end{equation*}%
of $A$ is inner. If $\mathbf{s},\mathbf{t}\in X$, then there is an element $%
z_{\mathbf{t},\mathbf{s}}$ in the center of the enveloping von Neumann
algebra of $A$ such that%
\begin{equation*}
\mathrm{\mathrm{exp}}\left( i\sum_{n\in \mathbb{N}}(t_{n}-s_{n})x_{n}+z_{%
\mathbf{t},\mathbf{s}}\right)
\end{equation*}%
multiplies $A$. Hence%
\begin{equation*}
y_{\mathbf{t},\mathbf{s}}=\mathrm{\mathrm{exp}}\left( i\sum_{n\in \mathbb{N}%
}(t_{n}-s_{n})x_{n}+z_{\mathbf{t},\mathbf{s}}\right) b_{0}
\end{equation*}%
is an element of $A$. By Lemma \ref{Lemma: uncountable}, one can find an
uncountable subset $Y$ of $X$ such that, for any $\mathbf{t},\mathbf{s}\in Y$%
, there is $m\in \mathbb{N}$ such that%
\begin{equation*}
\left\vert t_{m}-s_{m}\right\vert \geq \frac{1}{4}\text{.}
\end{equation*}%
Fix $\mathbf{s}\in Y$ and observe that, for $\mathbf{t},\mathbf{t}^{\prime
}\in Y$,%
\begin{equation*}
\pi _{n_{0}}(\mathrm{exp}(z_{\mathbf{t}^{\prime },\mathbf{s}}-z_{\mathbf{t},%
\mathbf{s}}))=\mu 1
\end{equation*}%
is a scalar multiple of the identity. Therefore%
\begin{eqnarray*}
&&\left\Vert y_{\mathbf{t},\mathbf{s}}-y_{\mathbf{t}^{\prime },\mathbf{s}%
}\right\Vert \\
&=&\left\Vert \left( \mathrm{exp}\left( i\sum_{n\in \mathbb{N}%
}(t_{n}-t_{n}^{\prime })x_{n}\right) -\mathrm{exp}(z_{\mathbf{t}^{\prime },%
\mathbf{s}}-z_{\mathbf{t,s}})\right) a_{0}\right\Vert \\
&\geq &\left\Vert \pi _{n_{0}}\left( \left( \mathrm{exp}\left( i\sum_{n\in 
\mathbb{N}}(t_{n}-t_{n}^{\prime })x_{n}\right) -\mathrm{exp}(z_{\mathbf{t}%
^{\prime },\mathbf{s}}-z_{\mathbf{t,s}})\right) a_{0}\right) \right\Vert \\
&=&\left\Vert \pi _{n_{0}}((\mathrm{exp}((t_{n_{0}}-t_{n_{0}}^{\prime
})x_{n})-\mu )a_{0})\right\Vert \\
&\geq &\varepsilon \text{.}
\end{eqnarray*}%
This contradicts the separability of $A$.
\end{proof}

The proofs of Lemma \ref{Lemma: quotient} and Lemma \ref{Lemma: spectrum
omega+1} are contained, respectively, in the proofs of Lemmas 3.6 and 3.7 of 
\cite{Akemann-Pedersen} and in the proof of the implication $\left( i\right)
\Rightarrow \left( ii\right) $ at page 139 of \cite{Elliott3}.

Recall that a point $x$ of a topological space $X$ is called \textit{%
separated }if, given any point $y$ of $X$ distinct from $x$, there are
disjoint open neighborhoods of $x$ and $y$.

\begin{lemma}
\label{Lemma: quotient}Suppose that $A$ is a C*-algebra whose primitive
spectrum $\check{A}$ is $T_{1}$. Consider a sequence $\left( \xi _{n}\right)
_{n\in \mathbb{N}}$ of separated points in $\check{A}$. Define $F$ to be the
set of limit points of the sequence $\left( \xi _{n}\right) _{n\in \mathbb{N}%
}$ and $I$ to be the closed self-adjoint ideal of $A$ corresponding to $F$.
If either the quotient $A\left/ I\right. $ does not have continuous trace,
or the multiplier algebra of $A\left/ I\right. $ has nontrivial center, then 
$A$ has Property AEP$^{+}$.
\end{lemma}

\begin{lemma}
\label{Lemma: spectrum omega+1}If $A$ is a C*-algebra whose spectrum $\hat{A}
$ is homeomorphic to the one-point compactification of a countable discrete
space, then $A$ has Property AEP$^{+}$.
\end{lemma}

We can now prove the main result of this section that Property AEP as
defined in \ref{Definition: property AEP} is equivalent to having an outer
*-derivation.

\begin{theorem}
\label{Theorem: AEP}If $A$ is a C*-algebra, the following statements are
equivalent:

\begin{enumerate}
\item $A$ has an outer derivation;

\item $A$ has Property AEP.
\end{enumerate}
\end{theorem}

\begin{proof}
We have already pointed out that Property AEP implies the existence of an
outer *-derivation. It remains only to show the converse. Suppose that $A$
has an outer derivation. By Lemma 16 of \cite{Elliott3}, either there is a
quotient $B$ of $A$ whose spectrum $\hat{B}$ is homeomorphic to the one
point compactification of a countable discrete space, or the primitive
spectrum $\check{A}$ of $A$ is not Hausdorff. In the first case, $A$ has
Property AEP by virtue of Lemma \ref{Lemma: spectrum omega+1} and Lemma \ref%
{Lemma: lifting}. Suppose that, instead, the primitive spectrum $\check{A}$
of $A$ is not Hausdorff. If $\check{A}$ is not even $T_{1}$, the conclusion
follows from Lemma \ref{Lemma: non t1}. Suppose now that $\check{A}$ is $%
T_{1}$. Since $\check{A}$ is not Hausdorff, there are two points $\rho
_{0},\rho _{1}$ of $\check{A}$ that do not admit any disjoint open
neighbourhoods. By separability, and since separated points are dense in $%
\check{A}$ by Proposition 1 of \cite{Dixmier-separated}, one can find a
sequence $\left( \xi _{n}\right) _{n\in \mathbb{N}}$ of separated points of $%
\check{A}$ whose set $F$ of limit points contains both $\rho _{0}$ and $\rho
_{1}$. Define $I$ to be the closed self-adjoint ideal $I$ of $A$
corresponding to the closed subset $F$ of $\check{A}$. By Lemma 3.1 of \cite%
{Akemann-Pedersen}, either $A\left/ I\right. $ does not have continuous
trace or the multiplier algebra of $A\left/ I\right. $ has nontrivial
center. In either cases, it follows that $A$ has Property AEP$^{+}$ and, in
particular, the weaker Property AEP by Lemma \ref{Lemma: quotient}.
\end{proof}

\section{The case of algebras with only inner derivations}

\label{Section: inner}In this section we will prove that, if a C*-algebra $A$
with only inner derivations does not have continuous trace, then the
relation of unitary equivalence of approximately inner automorphisms of $A$
is not classifiable by countable structures. In proving this fact we will
also show that any such C*-algebra contains a central sequence that is not
hypercentral.

If $A$ is a C*-algebra, denote by $A^{\infty }$ the quotient of the direct
product $\prod_{n\in \mathbb{N}}A$ by the direct sum $\bigoplus_{n\in 
\mathbb{N}}A$ (defined as in \cite{Blackadar} II.8.1.2.). Identifying as it
is customary $A$ with the algebra of elements of $A^{\infty }$ admitting
constant representative sequence, denote by $A_{\infty }$ the relative
commutant 
\begin{equation*}
A^{\prime }\cap A^{\infty }=\left\{ x\in A^{\infty }\mid \forall y\in A\text{%
, }[x,y]=0\right\} \text{.}
\end{equation*}%
Finally define%
\begin{equation*}
\mathrm{Ann}(A,A_{\infty })=\left\{ x\in A_{\infty }\mid \forall y\in A\text{%
, }xy=0\right\}
\end{equation*}%
to be the annihilator ideal of $A$ in $A_{\infty }$. Observe that, if $A$ is
unital, then $\mathrm{Ann}(A,A_{\infty })$ is the trivial ideal.

A \textbf{central sequence} in a C*-algebra $A$ is a bounded sequence $%
\left( x_{n}\right) _{n\in \mathbb{N}}$ of elements of $A$ that
asymptotically commute with any element of $A$. This means that for any $%
a\in A$,%
\begin{equation*}
\lim_{n\rightarrow +\infty }\left\Vert \left[ x_{n},a\right] \right\Vert =0%
\text{.}
\end{equation*}%
Equivalently the image of $\left( x_{n}\right) _{n\in \mathbb{N}}$ in the
quotient of $\prod_{n\in \mathbb{N}}A$ by $\bigoplus_{n\in \mathbb{N}}A$
belongs to $A_{\infty }$. From the last characterization it is clear that if 
$\left( x_{n}\right) _{n\in \mathbb{N}}$ is a central sequence of elements $%
A $ with spectra contained in some subset $D$ of $\mathbb{C}$ and $%
f:D\rightarrow \mathbb{C}$ is a continuous function such that $f\left(
0\right) =0$, then the sequence $\left( f\left( x_{n}\right) \right) _{n\in 
\mathbb{N}}$ is central. It is not difficult to verify that, if $\left(
x_{n}\right) _{n\in \mathbb{N}}$ is a central sequence and $b\in M(A)$, then
the sequence $\left( \left[ b,x_{n}\right] \right) _{n\in \mathbb{N}}$
converges strictly to $0$.

If $A$ is unital, a central sequence $\left( x_{n}\right) _{n\in \mathbb{N}}$
is called \textbf{hypercentral }(see \cite{Phillips2} Section 1) if it
asymptotically commutes with any other central sequence. This amounts to say
that for any other central sequence $\left( y_{n}\right) _{n\in \mathbb{N}}$%
\begin{equation*}
\lim_{n\rightarrow +\infty }\left\Vert \left[ x_{n},y_{n}\right] \right\Vert
=0\text{.}
\end{equation*}%
Equivalently the image of $\left( x_{n}\right) _{n\in \mathbb{N}}$ in the
quotient of $\prod_{n\in \mathbb{N}}A$ by $\bigoplus_{n\in \mathbb{N}}A$
belongs to the center of $A_{\infty }$. In order to generalize the notion of
hypercentral sequence to the nonunital setting it is convenient for our
purposes to consider the strict topology rather than the norm topology.
Henceforth we give the following definition:

\begin{definition}
If $A$ is a (not necessarily unital) C*-algebra, a sequence $\left(
x_{n}\right) _{n\in \mathbb{N}}$ of elements of $A$ is called \textbf{%
hypercentral} if it is central and, for any other central sequence $\left(
y_{n}\right) _{n\in \mathbb{N}}$, the sequence%
\begin{equation*}
\left( \left[ x_{n},y_{n}\right] \right) _{n\in \mathbb{N}}
\end{equation*}%
converges to $0$ in the strict topology.
\end{definition}

Observe that a central sequence $\left( x_{n}\right) _{n\in \mathbb{N}}$ is
hypercentral if and only if the image of the element of $A_{\infty }$ having 
$\left( x_{n}\right) _{n\in \mathbb{\mathbb{N}}}$ as representative sequence
in the quotient $A_{\infty }\left/ \mathrm{Ann}(A,A_{\infty })\right. $
belongs to the center of $A_{\infty }\left/ \mathrm{Ann}(A,A_{\infty
})\right. $. It is clear from this characterization that, if $\left(
x_{n}\right) _{n\in \mathbb{N}}$ is a hypercentral sequence of elements of $%
A $ with spectra contained in some subset $D$ of $\mathbb{C}$ and $%
f:D\rightarrow \mathbb{C\ } $ is a complex-valued continuous function such
that $f\left( 0\right) =0$, then the sequence $\left( f\left( x_{n}\right)
\right) _{n\in \mathbb{N}}$ is hypercentral. When $A$ is unital the ideal $%
\mathrm{Ann}\left( A,A_{\infty }\right) $ is trivial, and hence this
definition agrees with the usual definition of hypercentral sequence.

The fact that a unital simple C*-algebra contains a central sequence that is
not hypercentral is a particular case of Theorem 3.6 of \cite{Phillips2}. We
will show here how one can generalize this fact to all simple nonelementary
C*-algebras. The proof deeply relies on ideas from \cite{Phillips2}.

\begin{lemma}
\label{Lemma: converges strictly}If $\left( x_{n}\right) _{n\in \mathbb{N}}$
is a hypercentral sequence in $A$ and $\alpha $ is an approximately inner
automorphism of $A$, then $\left( \alpha (x_{n})-x_{n}\right) _{n\in \mathbb{%
N}}$ converges strictly to $0$.
\end{lemma}

\begin{proof}
Suppose that $\varepsilon >0$ and $a$ is an element of $A$. Since $\left(
x_{n}\right) _{n\in \mathbb{N}}$ is a hypercentral sequence, by strict
density of the unit ball of $A$ in the unit ball of $M(A)$ (see \cite%
{Blackadar} II.7.3.11 and \cite{Lance} Proposition 1.4) there is a finite
subset $F$ of the unit ball of $A$, a positive real number $\delta $, and a
natural number $n_{0}$ such that, for every $n\geq n_{0}$ and every $y$ in
the unit ball $M(A)$ such that $\left\Vert \left[ y,z\right] \right\Vert
<\delta $ for every $z\in F$,%
\begin{equation*}
\max \left\{ \left\Vert a(x_{n}y-yx_{n})\right\Vert ,\left\Vert
(x_{n}y-yx_{n})a\right\Vert \right\} \leq \varepsilon \text{.}
\end{equation*}%
Consider the open neighbourhood%
\begin{equation*}
U=\left\{ \alpha \in \mathrm{Aut}(A)\mid \left\Vert \alpha \left( x\right)
-x\right\Vert <\delta \text{ for every }x\in F\right\}
\end{equation*}%
of $id_{A}$ in $\mathrm{Aut}(A)$. Observe that if $\beta \in U$ is inner,
then for every $n\geq n_{0}$%
\begin{equation*}
\left\Vert (\beta (x_{n})-x_{n})a\right\Vert \leq \varepsilon
\end{equation*}%
and%
\begin{equation*}
\left\Vert a(\beta (x_{n})-x_{n})\right\Vert \leq \varepsilon \text{.}
\end{equation*}%
Approximating with inner automorphisms, one can see that the same is true if 
$\beta \in U$ is just approximately inner. Since $\alpha $ is approximately
inner, there is a unitary multiplier $u$ of $A$ and an approximately inner
automorphism $\beta $ of $A$ in $U$ such that%
\begin{equation*}
\alpha =\beta \circ \mathrm{Ad}(u)\text{.}
\end{equation*}%
Consider a natural number $n_{1}\geq n_{0}$ such that, for $n\geq n_{1}$,%
\begin{equation*}
\left\Vert \beta ^{-1}(a)[x_{n},u]\right\Vert \leq \varepsilon
\end{equation*}%
and%
\begin{equation*}
\left\Vert \lbrack x_{n},u^{\ast }]\beta ^{-1}(a)\right\Vert \leq \varepsilon
\end{equation*}%
It follows that, if $n\geq n_{1}$,%
\begin{eqnarray*}
\left\Vert a(\alpha (x_{n})-x_{n})\right\Vert &\leq &\left\Vert a\beta (%
\mathrm{Ad}(u)\left( x_{n}\right) -x_{n})\right\Vert +\left\Vert \beta
(x_{n})-x_{n}\right\Vert \\
&\leq &\left\Vert \beta ^{-1}(a)(ux_{n}u^{\ast }-x_{n})\right\Vert
+\varepsilon \\
&=&\left\Vert \beta ^{-1}(a)[x_{n},u]\right\Vert +\varepsilon \\
&\leq &2\varepsilon
\end{eqnarray*}%
and, analogously,%
\begin{equation*}
\left\Vert (\alpha (x_{n})-x_{n})a\right\Vert \leq 2\varepsilon \text{.}
\end{equation*}%
Since $\varepsilon $ was arbitrary, this concludes the proof of the fact that%
\begin{equation*}
\left( a(x_{n})-x_{n}\right) _{n\in \mathbb{N}}
\end{equation*}%
converges strictly to $0$.
\end{proof}

If $\alpha $ is an automorphism of a C*-algebra $A$, then $\alpha ^{\ast
\ast }$ denotes the unique extension of $\alpha $ to a $\sigma $-weakly
continuous automorphism of the enveloping von Neumann algebra $A^{\ast \ast
} $ of $A$ (defined as in Proposition III.5.2.10 of \cite{Blackadar}).

\begin{lemma}
\label{Lemma: fixing pointwise}Suppose that $A$ is a C*-algebra such that
every central sequence in $A$ is hypercentral. If $\alpha $ is an
approximately inner automorphism of $A$, then $\alpha ^{\ast \ast }$ fixes
pointwise the center of $A^{\ast \ast }$, i.e.\ $\alpha ^{\ast \ast }\left(
z\right) =z$ for every central element of $A^{\ast \ast }$.
\end{lemma}

\begin{proof}
Observe that $z$ derives $A$, since%
\begin{equation*}
za-az=0\in A
\end{equation*}%
for every $a\in A$. Thus, by Lemma 1.1 of \cite{Akemann-Pedersen}, there is
a bounded net $\left( z_{\lambda }\right) $ in $A$ converging strongly to $z$
such that, for every $a\in A$,%
\begin{equation*}
\lim_{\lambda }\left\Vert \left[ z_{\lambda }-z,a\right] \right\Vert =0\text{%
.}
\end{equation*}%
Recall that strong and $\sigma $-strong topology agree on bounded sets, and
that the $\sigma $-strong topology is stronger than the $\sigma $-weak
topology (see Definition I.3.1.1 of \cite{Blackadar}). Thus the net $\left(
z_{\lambda }\right) $ converges \textit{a fortiori} $\sigma $-weakly to $z$.
Since the $\sigma $-weak topology on $A^{\ast \ast }$ is the weak* topology
on $A^{\ast \ast }$ regarded as the dual space of $A^{\ast }$, the unit ball
of $A^{\ast \ast }$ is $\sigma $-weakly compact by Alaoglu's theorem
(Theorem 2.5.2 in \cite{Analysis_now}). Moreover by Kaplanski's Density
Theorem (Theorem 2.3.3 in \cite{Pedersen}) the unit ball of $A$ is $\sigma $%
-weakly dense in the unit ball of $A^{\ast \ast }$. As a consequence the
unit ball of $A^{\ast \ast }$ is $\sigma $-weakly metrizable, and the same
holds for balls of arbitrary radius Thus we can find a bounded \textit{%
sequence} $\left( z_{n}\right) _{n\in \mathbb{N}}$ in $A$ converging $\sigma 
$-weakly to $z$ such that, for every $a\in A$,%
\begin{equation*}
\lim_{n\rightarrow +\infty }\left\Vert \left[ z_{n}-z,a\right] \right\Vert =0%
\text{.}
\end{equation*}%
Since%
\begin{equation*}
\left[ z_{n}-z,a\right] =\left[ z_{n},a\right]
\end{equation*}%
for every $n\in \mathbb{N}$, $\left( z_{n}\right) _{n\in \mathbb{N}}$ is a
central and hence hypercentral sequence (every central sequence of $A$ is
hypercentral by assumption). Being $\alpha ^{\ast \ast }$ a $\sigma $-weakly
continuous automorphism of $A^{\ast \ast }$ extending $\alpha $, $\left(
\alpha (z_{n})\right) _{n\in \mathbb{N}}$ converges $\sigma $-weakly to $%
\alpha ^{\ast \ast }(z)$. It follows from Lemma \ref{Lemma: fixing pointwise}
and from the facts that $\alpha $ is approximately inner and the sequence $%
\left( z_{n}\right) _{n\in \mathbb{N}}$ is hypercentral that the bounded
sequence $\left( z_{n}-\alpha (z_{n})\right) _{n\in \mathbb{N}}$ converges
strictly to $0$. By Lemma 1.3.1 of \cite{Loring} and since weak and $\sigma $%
-weak topology agree on bounded sets, the sequence $\left( z_{n}-\alpha
(z_{n})\right) _{n\in \mathbb{N}}$ converges $\sigma $-weakly to $0$.
Therefore $z=\alpha ^{\ast \ast }\left( z\right) $.
\end{proof}

A C*-algebra is called \textit{elementary} if it is *-isomorphic to the
algebra of compact operators on some Hilbert space (see Definition IV.1.2.1
in \cite{Blackadar}). By Corollary 1 of Theorem 1.4.2 in \cite{Arveson} any
elementary C*-algebra is simple. Moreover by Corollary 3 of Theorem 1.4.4 in 
\cite{Arveson} any automorphism of an elementary C*-algebra is inner; in
particular the group $\mathrm{Inn}(A)$ of inner automorphisms of an
elementary C*-algebra $A$ is closed inside the group $\mathrm{Aut}(A)$ of
all automorphisms. Conversely if the group of inner automorphisms of a
simple C*-algebra $A$ is closed, then $A$ is elementary by Theorem 3.1 of 
\cite{Phillips1} together with Corollary IV.1.2.6 and Proposition IV.1.4.19
of \cite{Blackadar}.

Recall that in this paper all C*-algebras (apart from multiplier algebras
and enveloping von Neumann algebras) are assumed to be norm separable. In
particular separability of $A$ is assumed in Proposition \ref{Proposition:
simple hypercentral}; however we do not know if the separability assumption
is necessary there.

\begin{proposition}
\label{Proposition: simple hypercentral}If $A$ is a simple C*-algebra such
that every central sequence in $A$ is hypercentral, then $A$ is elementary.
\end{proposition}

\begin{proof}
It is enough to show that $\mathrm{Inn}\left( A\right) $ is closed in $%
\mathrm{Aut}\left( A\right) $ or, equivalently, that no outer automorphism
is approximately inner. Fix an outer automorphism $\alpha $ of $A$. Since $A$
is simple, by \cite{Kishimoto2} Corollary 2.3, there is an irreducible
representation $\pi $ such that $\pi $ and $\pi \circ \alpha $ are not
equivalent. If $z$ is the central cover of $\pi $ in $A^{\ast \ast }$
(defined as in \cite{Pedersen} 3.8.1), then $\alpha ^{\ast \ast }(z)$ is the
central cover of $\pi \circ \alpha $; moreover, being $\pi $ and $\pi \circ
\alpha $ not equivalent, $\alpha ^{\ast \ast }(z)$ is different from $z$ by
Theorem 3.8.2 of \cite{Pedersen}. Thus $\alpha ^{\ast \ast }$ does not fixes
pointwise the center of $A^{\ast \ast }$ and, by Lemma \ref{Lemma: fixing
pointwise}, $\alpha $ is not approximately inner.
\end{proof}

Proposition \ref{Proposition: simple hypercentral} shows that any simple
nonelementary C*-algebra contains a central sequence that is not
hypercentral. It is clear that the same conclusion holds for any C*-algebra
containing a simple nonelementary C*-algebra as a direct summand. By Theorem
3.9 of \cite{Akemann-Pedersen}, this class of C*-algebras coincides with the
class of C*-algebras that have only inner derivations and do not have
continuous trace. This concludes the proof of the following proposition:

\begin{proposition}
\label{Proposition: existence hypercentral}If $A$ is a C*-algebra that does
not have continuous trace and has only inner derivations, then $A$ contains
a central sequence that is not hypercentral.
\end{proposition}

In view of this result, in order to finish the proof of Theorem \ref%
{Theorem: main}, it is enough to show that its conclusion holds for a
C*-algebra $A$ containing a central sequence that is not hypercentral.

\begin{proposition}
\label{Proposition: fundamental hypercentral}If $A$ is a C*-algebra
containing a central sequence that is not hypercentral, then the
approximately inner automorphisms of $A$ are not classifiable by countable
structures up to unitary equivalence.
\end{proposition}

\begin{proof}
Fix a dense sequence $\left( a_{n}\right) _{n\in \mathbb{N}}$ in the unit
ball of $A$. Suppose that $\left( x_{n}\right) _{n\in \mathbb{N}}$ is a
central sequence in $A$ that is not hypercentral. Thus there is a central
sequence $\left( y_{n}\right) _{n\in \mathbb{N}}$ in $A$ such that the
sequence%
\begin{equation*}
\left( \left[ x_{n},y_{n}\right] \right) _{n\in \mathbb{N}}
\end{equation*}%
does not converge strictly to $0$. This implies that, for some positive
contraction $b$ in $A$, then the sequence%
\begin{equation*}
\left( b[x_{n},y_{n}]\right) _{n\in \mathbb{N}}
\end{equation*}%
does not converge to $0$ is norm. Without loss of generality we can assume
that, for every $n\in \mathbb{N}$, $x_{n}$ and $y_{n}$ are positive
contractions. Observe that the sequence $\left( \mathrm{exp}\left(
itx_{n}\right) -1\right) _{n\in \mathbb{N}}$ is not hypercentral for any $%
t\in \left( 0,1\right) $. After passing to subsequences, we can assume that
for some strictly positive real number $\varepsilon $, for every $t\in
\left( 0,1\right) $, every $s\in \left( \frac{1}{2},1\right) $, and every $%
n,k\in \mathbb{N}$ such that $k\leq n$:

\begin{itemize}
\item $\left\Vert \left[ a_{k},\mathrm{exp}(itx_{n})\right] \right\Vert
<2^{-n}$;

\item $\left\Vert b[x_{n},y_{n}]\right\Vert \geq \varepsilon $;

\item $\left\Vert b[\mathrm{exp}(isx_{n}),y_{n}]\right\Vert \geq \varepsilon 
$.
\end{itemize}

Define $\eta =\frac{\varepsilon }{20}$. After passing to a further
subsequence, we can assume that, for every $t\in \left( 0,1\right) $ and
every $n,k\in \mathbb{N}$ such that $k\leq n$:

\begin{itemize}
\item $\left\Vert \left[ \mathrm{exp}(itx_{k}),y_{n}\right] \right\Vert
<2^{-n}\eta $;

\item $\left\Vert \left[ y_{k},\mathrm{exp}(itx_{n})\right] \right\Vert
<2^{-n}\eta $;

\item $\left\Vert \left[ \mathrm{exp}(itx_{k}),\mathrm{\mathrm{\mathrm{exp}}}%
(isx_{n})\right] \right\Vert <2^{-n}\eta $.
\end{itemize}

It is not difficult to verify that, if $\mathbf{t}\in \left( 0,1\right) ^{%
\mathbb{N}}$, then the sequence%
\begin{equation*}
\left( \mathrm{Ad}{\left( \mathrm{\mathrm{\mathrm{\mathrm{\mathrm{\mathrm{%
\mathrm{exp}}}}}}}(it_{n}x_{n})\right) }\right) _{n\in \mathbb{N}}
\end{equation*}%
is Cauchy in $\mathrm{Aut}(A)$. Denoting by $f(\mathbf{t})$ its limit, one
obtains a function 
\begin{equation*}
f:\left( 0,1\right) ^{\mathbb{N}}\rightarrow \overline{\mathrm{Inn}(A)}\text{%
.}
\end{equation*}
In the rest of the proof we will show that $f$ satisfies the hypothesis of
Criterion \ref{Criterion: non-classifiability}. Suppose that $M$ is a
Lipschitz constant for the function $t\mapsto \mathrm{exp}(it)$ on $\left[
0,1\right] $. If $\mathbf{t},\mathbf{s}\in \left( 0,1\right) ^{\mathbb{N}}$
and $n\in \mathbb{N}$ are such that $\left\vert t_{k}-s_{k}\right\vert
<\varepsilon $ for $k\in \left\{ 1,2,\ldots ,n\right\} $, then it is easy to
see that%
\begin{equation*}
\left\Vert f(\mathbf{t})(a_{k})-f(\mathbf{s})(a_{k})\right\Vert \leq
2^{-n+1}+\varepsilon M
\end{equation*}%
for $k\leq n$. This shows that the function $f\ $is continuous,
particularly, Baire measurable. Moreover, if $\mathbf{t},\mathbf{s}\in
\left( 0,1\right) ^{\mathbb{N}}$ are such that $\mathbf{s}-\mathbf{t}\in
\ell ^{1}$, then the sequence%
\begin{equation*}
\left( \mathrm{exp}(it_{1}x_{1})\cdots \mathrm{exp}(it_{n}x_{n})\mathrm{exp}%
(-is_{n}x_{n})\cdots \mathrm{exp}(-is_{1}x_{1})\right) _{n\in \mathbb{N}}
\end{equation*}%
is Cauchy in $U(A)$, and hence has a limit $u\in U(A)$. It is now readily
verified that%
\begin{equation*}
f(\mathbf{t})=\mathrm{Ad}(u)\circ f(\mathbf{s})
\end{equation*}%
and hence $f(\mathbf{t})$ and $f(\mathbf{s})$ are unitarily equivalent.
Finally, suppose that $C$ is a comeager subset of $\left( 0,1\right) ^{%
\mathbb{N}}$. Thus, there are $\mathbf{t},\mathbf{s}\in C$ such that $%
\left\vert t_{n}-s_{n}\right\vert \in \left( \frac{1}{2},1\right) $ for
infinitely many $n\in \mathbb{N}$. We claim that $f(\mathbf{t})$ and $f(%
\mathbf{s})$ are not unitarily equivalent. Suppose by contradiction that
this is not the case. Thus there is $u\in U(A)$ such that%
\begin{equation*}
f(\mathbf{t})=\mathrm{Ad}(u)\circ f(\mathbf{s})\text{.}
\end{equation*}%
This implies that the sequence%
\begin{equation*}
\left( u\mathrm{exp}(it_{1}x_{1})\cdots \mathrm{exp}(it_{n}x_{n})\mathrm{exp}%
(-is_{n}x_{n})\cdots \mathrm{exp}(-is_{1}x_{1})\right) _{n\in \mathbb{N}}
\end{equation*}%
in $U(A)$ is central, i.e.\ asymptotically commutes (in norm) with any
element of $A$. Fix now any $n_{0}\in \mathbb{N}$ such that $\left\vert
t_{n_{0}}-s_{n_{0}}\right\vert \in \left( \frac{1}{2},1\right) $ and%
\begin{equation*}
\left\Vert b[y_{n},u]\right\Vert <\eta
\end{equation*}%
for $n\geq n_{0}$. Suppose that $n>n_{0}$. Using the fact that the elements $%
\mathrm{exp}\left( it_{m}x_{m}\right) $ and $\mathrm{exp}\left(
it_{k}x_{k}\right) $ commute up to $\eta 2^{-m}$ for $k,m\in \mathbb{N}$,
one can show that%
\begin{equation*}
by_{n_{0}}u\mathrm{exp}(it_{1}x_{1})\cdots \mathrm{exp}(it_{n}x_{n})\mathrm{%
exp}(-is_{n}x_{n})\cdots \mathrm{exp}(-is_{1}x_{1})
\end{equation*}%
is at distance at most $5\eta $ from 
\begin{eqnarray*}
&&buy_{n_{0}}\mathrm{exp}(i(t_{n_{0}}-s_{n_{0}})x_{n_{0}})\mathrm{exp}%
(it_{1}x_{1})\cdots \widehat{\mathrm{exp}(it_{n_{0}}x_{n_{0}})} \\
&&\cdots \mathrm{\mathrm{ex}}(it_{n}x_{n})\mathrm{exp}(-is_{n}x_{n})\cdots 
\widehat{\mathrm{exp}(is_{n_{0}}x_{n_{0}})}\cdots \mathrm{exp}(-is_{1}x_{1})%
\text{,}
\end{eqnarray*}%
where $\widehat{\mathrm{exp}(it_{n_{0}}x_{n_{0}})}$ and $\widehat{\mathrm{exp%
}(is_{n_{0}}x_{n_{0}})}$ indicate omitted terms in the product. Similarly
\begin{equation*}
bu\mathrm{exp}(it_{1}x_{1})\cdots \mathrm{exp}(it_{n}x_{n})\mathrm{exp}%
(-is_{n}x_{n})\cdots \mathrm{exp}(-is_{1}x_{1})y_{n_{0}}
\end{equation*}%
is at distance at most $5\eta $ from 
\begin{eqnarray*}
&&bu\mathrm{exp}\left( i\left( t_{n_{0}}-s_{n_{0}}\right) x_{n_{0}}\right)
y_{n_{0}}\mathrm{exp}\left( it_{1}x_{1}\right) \cdots \widehat{\mathrm{exp}%
\left( it_{n_{0}}x_{n_{0}}\right) } \\
&&\cdots \mathrm{exp}\left( it_{n}x_{n}\right) \mathrm{exp}\left(
-is_{n}x_{n}\right) \cdots \widehat{\mathrm{exp}\left(
is_{n_{0}}x_{n_{0}}\right) }\ldots \mathrm{exp}\left( -is_{1}x_{1}\right) 
\text{.}
\end{eqnarray*}%
Thus, the norm of the commutator of 
\begin{equation*}
u\mathrm{exp}(it_{1}x_{1})\cdots \mathrm{exp}(it_{n}x_{n})\mathrm{exp}%
(-is_{n}x_{n})\cdots \mathrm{exp}(-is_{1}x_{1})
\end{equation*}%
and $y_{0}$ is at least%
\begin{equation*}
\left\Vert b[\mathrm{exp}(i(t_{n_{0}}-s_{n_{0}})x_{n_{0}}),y_{n_{0}}]\right%
\Vert -10\eta \geq \varepsilon -10\eta \geq \frac{\varepsilon }{2}\text{.}
\end{equation*}%
This contradicts the fact that the sequence%
\begin{equation*}
\left( u\mathrm{exp(}it_{1}x_{1})\cdots \mathrm{exp}(it_{n}x_{n})\mathrm{exp}%
(-is_{n}x_{n})\cdots \mathrm{exp}\left( -is_{1}x_{1}\right) \right) _{n\in 
\mathbb{N}}
\end{equation*}%
is central and concludes the proof.
\end{proof}

\section{A dichotomy for derivations}

\label{Section: dichotomy derivations}

If $A$ is a C*-algebra, then we denote as in Section \ref{Section: outer} by 
$\Delta _{0}(A)$ the separable Banach space of inner derivations of $A$
endowed with the norm $\left\Vert \cdot \right\Vert _{\Delta _{0}\left(
A\right) }$ and by $\overline{\Delta _{0}\left( A\right) }$ the closure of $%
\Delta _{0}\left( A\right) $ inside the space $\Delta \left( A\right) $ of
derivations of $A$ endowed with the operator norm. Suppose that $E_{\Delta
(A)}$ is the Borel equivalence relation on $\overline{\Delta _{0}(A)}$
defined by%
\begin{equation*}
\delta _{0}E_{\Delta (A)}\delta _{1}\quad \text{iff}\quad \delta _{0}-\delta
_{1}\in \Delta _{0}(A)\text{.}
\end{equation*}%
Observe that $E_{\Delta (A)}$ is the orbit equivalence relation associated
with the continuous action of the additive group of $\Delta _{0}(A)$ on $%
\overline{\Delta _{0}(A)}$ by translation.

\begin{theorem}
\label{Theorem: dichotomy derivations}If $A$ is a unital C*-algebra, then
the following statements are equivalent:

\begin{enumerate}
\item $\Delta _{0}(A)$ is closed in $\Delta (A)$;

\item $E_{\Delta (A)}$ is smooth;

\item $E_{\Delta (A)}$ is classifiable by countable structures;

\item $A$ has continuous trace.
\end{enumerate}
\end{theorem}

The equivalence of $1$ and $4$ follows from Theorem 5.3 of \cite%
{Kadison-Lance-Ringrose} together with the equivalence of $1$ and $3$ in
Theorem \ref{Theorem: unital}. The implication $1\Rightarrow 2$ follows from
Exercise 4.4 of \cite{Gao}. Trivially $2$ is stronger than $3$. Finally
observe that $\Delta _{0}(A)$ and $\overline{\Delta _{0}(A)}$ satisfy the
hypothesis of Lemma 2.1 of \cite{Sas-To:Araki}. In fact, as discussed at the
beginning of Section \ref{Section: inner}, $\Delta _{0}(A)$ endowed with the
norm%
\begin{equation*}
\left\Vert \mathrm{ad}\left( ia\right) \right\Vert _{\Delta (A)}=\inf
\left\{ \left\Vert a-z\right\Vert \mid z\in A^{\prime }\cap A\right\}
\end{equation*}%
is a separable Banach space. Moreover the inclusion map of $\Delta _{0}(A)$
in $\overline{\Delta _{0}(A)}\subset \Delta (A)$ is bounded with norm at
most $2$. Thus, if $\Delta _{0}(A)$ is not closed in $\Delta (A)$, then the
continuous action of the additive group $\Delta _{0}(A)$ on $\overline{%
\Delta _{0}(A)}$ by translation is turbulent. Hjorth's turbulence theorem
recalled at the beginning of Section \ref{Section: criteria} concludes the
proof of the implication $3\Rightarrow 1$.

\section{Questions}

\label{Section: open problems}

As pointed out in Section \ref{Section: Introduction}, the implication $%
3\Rightarrow 1$ of Theorem \ref{Theorem: main} does not hold in general.
Remark 0.9 of \cite{Raeburn-Rosenberg} provides an example of a C*-algebra $%
A $ that has continuous trace such that the group $\mathrm{Inn}(A)$ of inner
automorphisms of $A$ is not closed inside $\mathrm{Aut}(A)$. This implies
that the automorphisms of $A$ are not concretely classifiable up to unitary
equivalence. It would be interesting to know if the automorphisms of $A$ are
at least classifiable by countable structures up to unitary equivalence.
More generally we would like to suggest the following question:

\begin{question}
\label{Question: continuous trace}Is there a C*-algebra $A$ such that the
automorphisms of $A$ are classifiable by countable structures but not
concretely classifiable?
\end{question}

By Theorem \ref{Theorem: main} and the discussion preceding Theorem \ref%
{Theorem: unital}, such C*-algebra would necessarily have continuous trace
and spectrum not homotopically equivalent to a compact space. It is clear
that Question \ref{Question: continuous trace} has negative answer if and
only the dichotomy expressed by the equivalence of $1$ and $2$ in Theorem %
\ref{Theorem: unital} holds for any (not necessarily unital) C*-algebra.

It would also be interesting to study the Borel complexity of the
equivalence relation of conjugacy inside the automorphism group $\mathrm{Aut}%
(A)$ of a C*-algebra $A$. Recall that two automorphisms $\alpha ,\beta $ of $%
A$ are \textit{conjugate} if there is a third automorphism $\gamma $ of $A$
such that $\alpha =\gamma \circ \beta \circ \gamma ^{-1}$. Observe that this
is the orbit equivalence relation associated with the action of $\mathrm{Aut}%
(A)$ on itself by conjugation.

It is worth noting that a dichotomy result as in Theorem \ref{Theorem:
unital} does not hold for the equivalence relation of conjugacy even for
unital commutative C*-algebras. If $X$ is a compact metrizable space, denote
by $C(X)$ the unital commutative C*-algebra of complex-valued continuous
functions on $X$ (a classic result of Gelfand and Naimark asserts that any
unital commutative C*-algebra is of this form, see Theorem II.2.2.4 of \cite{Blackadar}).
Observe that by II.2.2.5 of \cite{Blackadar} the group $%
\mathrm{Aut}(C(X))$ of automorphisms of $C(X)$ is isomorphic as a Polish
group to the group $\mathrm{\mathrm{\mathrm{Hom}}eo}(X)$ of homeomorphisms
of $X$ endowed with the topology of pointwise convergence. Theorem 4.9 and
Corollary 4.11 of \cite{Hjorth-book} assert that the equivalence relation of
conjugacy inside $\mathrm{\mathrm{\mathrm{Hom}}eo}([0,1])$ is Borel complete
(see Definition 13.1.1 \cite{Gao}); in particular it is classifiable by
countable structures, but it is not smooth and not Borel. As a consequence
the same is true for the equivalence relation of conjugacy inside the
automorphism group of $C([0,1])$. An analogous result holds for the
automorphism group of $C(2^{\mathbb{N}})$ by Theorem 5 of \cite{Camerlo-Gao}%
. On the other hand the equivalence relation of conjugacy inside the
automorphism group of $C(\left[ 0,1\right] ^{2})$ is not classifiable by
countable structures by Theorem 4.17 of \cite{Hjorth-book}.

\subsection*{Acknowledgments}
We are extremely grateful to Ilijas Farah for his encouragement and support,
for many fundamental and insightful suggestions, and for his valuable
comments and feedback on a preliminary version of this paper. Moreover, we
would like to thank Nicolas Meffe for his grammatical and stylistic advice,
and Samuel Coskey, Kenneth Davidson, George Elliott, Thierry Giordano, Ilan
Hirshberg. David Kerr, Marcin Sabok, N. Christopher Phillips, Nicola Watson,
Stuart White, and Wilhelm Winter for several helpful conversations.


\begin{thebibliography}{99}
\bibitem{Akemann-Pedersen-Tomiyama} \textit{C.~A.~Akemann}, \textit{G.~A.~Elliott},
\textit{G.~K.~Pedersen}, and \textit{J.~Tomiyama}, Derivations and multipliers of C*-algebras, Amer. J. Math. \textbf{98} (1976), no.~3, 679--708.


\bibitem{Akemann-Pedersen} \textit{C.~A.~Akemann} and \textit{G.~K.~Pedersen}, {Central
sequences and inner derivations of separable C*-algebras}, Amer. J. Math. 
\textbf{101} (1979), 1047--1061. 

\bibitem{Arveson} \textit{W.~Arveson}, {An invitation to C*-algebras},
Springer-Verlag, New York, 1976. 

\bibitem{Blackadar} \textit{B.~Blackadar}, {Operator algebras}, Encyclopaedia of
Mathematical Sciences, vol. 122, Springer-Verlag, Berlin, 2006, Theory of
C*-algebras and von Neumann algebras, Operator Algebras and Non-commutative
Geometry, III. 


\bibitem{Becker-Kechris} \textit{H.~Becker} and \textit{A.~S.~Kechris}, {The Descriptive Set
Theory of Polish Group Actions}, London Mathematical Society Lecture Notes
Series 232, Cambridge University Press, 1996. 

\bibitem{Camerlo-Gao} \textit{R.~Camerlo}, \textit{S.~Gao}, {The completeness of the
isomorphism relation for countable Boolean algebras}, Transactions of the
American Mathematical Society 353 (2001), 491-518. 

\bibitem{Dixmier-separated} \textit{J.~Dixmier}, {Points séparés dans le spectre
d'une C*-algèbre}, Acta Sci. Math. Szeged \textbf{22} (1961), 115--128. 

\bibitem{Dixmier-traces} \textit{J.~Dixmier}, {Traces sur les C*-algèbres}, Ann. Inst.
Fourier (Grenoble) \textbf{13} (1963), no.~fasc. 1, 219--262. 

\bibitem{Elliott:classificationAF} \textit{G.~A. Elliott}, {On the classification of
inductive limits of sequences of semisimple finite-dimensional algebras}, J.
Algebra \textbf{38} (1976), no.~1, 29--44. %

\bibitem{Elliott3} \textit{G.~A. Elliott}, {Some C*-algebras with outer derivations. {III}}%
, Ann. Math. (2) \textbf{106} (1977), no.~1, 121--143. %

\bibitem{FTT1} \textit{I.~Farah}, \textit{A.~S.~Toms}, and \textit{A.~T{ö}rnquist}, {Turbulence, orbit
equivalence, and the classification of nuclear C*-algebras}, J.~Reine
Angew.~Math., to appear.

\bibitem{FTT2} \textit{I.~Farah}, \textit{A.~S.~Toms}, and \textit{A.~T{ö}rnquist}, {The descriptive set theory of C*-algebra invariants}%
, Int.~Math.~Res.~Notices, to appear.

\bibitem{Farah-Mackey-Borel} \textit{I.~Farah}, {A dichotomy for the {M}ackey {B}orel
structure}, Proceedings of the 11th {A}sian {L}ogic {C}onference, World
Sci.~Publ., Hackensack, NJ, 2012, pp.~86--93. 

\bibitem{Fell} \textit{J.~M.~G.~Fell}, {The structure of algebras of operator fields}%
, Acta Math. \textbf{106} (1961), 233--280. 

\bibitem{Gao} \textit{S.~Gao}, {Invariant descriptive set theory}, Pure and Applied
Mathematics (Boca Raton), vol. 293, CRC Press, Boca Raton, FL, 2009. 

\bibitem{Glimm} \textit{J.~Glimm}, {Type {I} C*-algebras}, Ann. of Math. (2) \textbf{%
73} (1961), 572--612. 

\bibitem{Hjorth-book} \textit{G.~Hjorth}, {Classification and orbit equivalence
relations}, Mathematical Surveys and Monographs, vol.~75, American
Mathematical Society, Providence, RI, 2000. 

\bibitem{Kadison-Lance-Ringrose} \textit{R.~V.~Kadison}, \textit{E.~C.~Lance},
and \textit{John~R. Ringrose}, {Derivations and automorphisms of operator algebras. {%
II}}, J. Functional Analysis \textbf{1} (1967), 204--221. %

\bibitem{Katsura-Matui} \textit{T.~Katsura} and \textit{H.~Matui}, {Classification of
uniformly outer actions of {$\mathbb{Z}^{2}$} on {UHF} algebras}, Adv. Math. 
\textbf{218} (2008), no.~3, 940--968. 

\bibitem{Kechris} \textit{A.~S.~Kechris}, {Classical descriptive set theory},
Graduate Texts in Mathematics, vol. 156, Springer-Verlag, New York, 1995. 

\bibitem{Kerr-Li-Pichot} \textit{D.~Kerr}, \textit{H.~Li}, and \textit{M.~Pichot}, {Turbulence,
representations, and trace-preserving actions}, Proc. Lond. Math. Soc. (3) 
\textbf{100} (2010), no.~2, 459--484. 

\bibitem{Kirchberg:congress} \textit{E.~Kirchberg}, {Exact C*-algebras, tensor
products, and the classification of purely infinite algebras}, Proceedings
of the {I}nternational {C}ongress of {M}athematicians, {V}ol.\ 1, 2 ({Z}ü%
rich, 1994) (Basel), Birkhäuser, 1995, pp.~943--954. 

\bibitem{Kishimoto2} \textit{A.~Kishimoto}, {Outer automorphisms and reduced crossed
products of simple C*-algebras}, Comm. Math. Phys. \textbf{81} (1981),
429-435. 

\bibitem{Kishimoto1} \textit{A.~Kishimoto}, {The {R}ohlin property for automorphisms of {%
UHF} algebras}, J. Reine Angew. Math. \textbf{465} (1995), 183--196. 

\bibitem{Lance} \textit{E.~C.~Lance}, {Hilbert C*-modules}, volume 210 of London
Mathematical Society Lecture Note Series. Cambridge University Press,
Cambridge, 1995.

\bibitem{Loring} \textit{T.~A.~Loring}, {Lifting solutions to perturbing problems in {%
C*-algebras}, Fields Institute Monographs, vol.~8, American Mathematical
Society, Providence, RI, 1997. 
}

\bibitem{Matui3} \textit{H.~Matui}, {Classification of outer actions of {$\mathbb{Z}%
^{N}$} on {$\mathcal{O}_{2}$}}, Adv. Math. \textbf{217} (2008), no.~6,
2872--2896. 

\bibitem{Matui2} \textit{H.~Matui}, {{$\mathbb{Z}^{N}$}-actions on {UHF} algebras of
infinite type}, J. Reine Angew. Math. \textbf{657} (2011), 225--244. %

\bibitem{Matui-Sato1} \textit{H.~Matui} and \textit{Y.~Sato}, {{$\mathcal{Z}$}-stability of
crossed products by strongly outer actions}, Comm. Math. Phys. \textbf{314}
(2012), no.~1, 193--228. 

\bibitem{Murphy} \textit{G.~J.~Murphy}, C*-algebras and operator theory. Academic
Press Inc., Boston, MA, 1990.

\bibitem{Nakamura} \textit{H.~Nakamura}, {Aperiodic automorphisms of nuclear purely
infinite simple C*-algebras}, Ergodic Theory Dynam. Systems \textbf{20}
(2000), no.~6, 1749--1765. 

\bibitem{Analysis_now} \textit{G.~K.~Pedersen}, {Analysis now}, Graduate Texts in
Mathematics, vol.~118, Springer-Verlag, New York, 1989. 

\bibitem{Pedersen} \textit{G.~K.~Pedersen}, {C*-algebras and their automorphism groups},
London Mathematical Society Monographs, vol.~14, Academic Press Inc.
[Harcourt Brace Jovanovich Publishers], London, 1979. %

\bibitem{Phillips1} \textit{J.~Phillips}, {Outer automorphisms of separable
C*-algebras}, J. Funct. Anal. \textbf{70} (1987), no.~1, 111--116. %

\bibitem{Phillips2} \textit{J.~Phillips}, {Central sequences and automorphisms of
C*-algebras}, Amer. J. Math. \textbf{110} (1988), no.~6, 1095--1117. %

\bibitem{Phillips-Raeburn} \textit{J.~Phillips} and \textit{I.~Raeburn}, {Automorphisms of
C*-algebras and second Cech cohomology}, Indiana Univ. Math. J. 29 (1980),
799-822. 

\bibitem{Phi:classification} \textit{N.C.~Phillips}, {A classification theorem for
nuclear purely infinite simple {C*-algebras}, Doc. Math. \textbf{5} (2000),
49--114 (electronic). }

\bibitem{Raeburn-Rosenberg} \textit{I.~Raeburn} and \textit{J.~Rosenberg}, {Crossed products
of continuous-trace {C*-algebras by smooth actions}, Trans. Amer. Math. Soc.
305 (1988), 1-45. 
}


\bibitem{Raeburn-Williams} \textit{I.~Raeburn} and \textit{D.~P.~Williams}, {Morita equivalence
and continuous-trace {C*-algebras}, Mathematical Surveys and Monographs,
Vol. 60, American Mathematical Society, Providence, RI, 1998. 
}


\bibitem{Sas-To:Araki} \textit{R.~Sasyk} and \textit{A.~Törnquist}, {Turbulence and Arai-Woods
factors} J. Funct. Anal. \textbf{259} (2010), no.~9, 2238--2252. %





\end{thebibliography}
\end{document}